\documentclass[]{amsart}

\pdfoutput=1

\usepackage[backref=true]{biblatex} 
\usepackage{graphicx}

\usepackage[colorinlistoftodos]{todonotes}

\usepackage{array}

\usepackage{Baskervaldx}

\usepackage{mathrsfs,mathrsfs,amsmath,amssymb,amsfonts,amsthm,centernot,tikz,tikz-qtree, tipa, mathtools, quiver} 

\usepackage[frenchmath, Baskervaldx]{newtxmath}

\usepackage{upgreek}

\usepackage{geometry}

\usepackage{bm}

\usepackage{hyperref}
\usepackage{cleveref}

\usetikzlibrary{cd}

\hypersetup{
    colorlinks=true,
    breaklinks=true,
    citecolor=violet,
    linkcolor=red,
    filecolor=magenta,      
    urlcolor=blue,
}

\DeclareMathOperator{\Div}{Div}

\DeclareMathOperator{\dee}{d}

\DeclareMathOperator{\Pic}{Pic}

\DeclareMathOperator{\Hom}{Hom}

\DeclareMathOperator{\id}{id}

\DeclareMathOperator{\im}{im}

\DeclareMathOperator{\sheafHom}{\mathscr{H}\text{\kern -3pt {\calligra\large om}}\,}

\DeclareMathOperator{\sgn}{sgn}

\newcommand{\Z}{\mathbb Z}

\newcommand{\R}{\mathbb R}

\newcommand{\OO}{\mathcal O}

\newcommand{\vphi}{\upvarphi}

\newcommand{\inv}{^{-1}}

\newcommand{\bs}{\backslash}

\newcommand{\Ab}{\text{\textbf{Ab}}}

\newcommand{\Cyc}{\text{\textbf{Cyc}}}

\newcommand{\orCyc}{\text{\textbf{OrCyc}}}

\newcommand{\op}{^{\text{op}}}

\newcommand{\mc}{\mathcal}

\newcommand{\mb}{\mathbb}

\theoremstyle{plain}

\newtheorem{theorem}{Theorem}[section]

\newtheorem{lemma}[theorem]{Lemma}

\newtheorem{prop}[theorem]{Proposition}

\newtheorem{cor}[theorem]{Corollary}

\newtheorem{conjecture}[theorem]{Conjecture}

\theoremstyle{definition}

\newtheorem{defn}[theorem]{Definition}

\theoremstyle{remark}

\newtheorem{example}[theorem]{Example}

\newtheorem{remark}[theorem]{Remark}

\newtheorem*{theorem*}{Theorem}

\theoremstyle{question}

\newtheorem*{question}{Question}

\title{A Torelli Theorem for Graph Isomorphisms}

\author{Sarah Griffith}

\begin{document}
\maketitle

\begin{abstract}
    It is known that isomorphisms of graph Jacobians induce cyclic bijections on the associated graphs. We characterize when such cyclic bijections can be strengthened to graph isomorphisms, in terms of an easily computed divisor. The result refines tools used in algebraic geometry to examine the fibers of the compactified Torelli map.
\end{abstract}

\section{Overview}
There is a known Torelli theorem for graphs, due to Caporaso-Viviani and Su-Wagner \cite{caporaso_torelli_2010} \cite{su_lattice_2010}, showing that under mild hypotheses, isomorphisms of Jacobians of graphs can be noncanonically lifted to isomorphisms of their associated graphic matroids. This result did not use any analogue of the theta divisor. A discrete analogue exists, and one might hope that an isomorphism of graph Jacobians preserving it has special properties. In this paper we show that preservation of discrete theta divisors yields a procedure for lifting isomorphisms of graph Jacobians not only to the matroid level but to the level of graph isomorphisms. Our main result is \cref{geometric rigidity} and its corollary, which describe this lifting. These are restated in a form explicitly connected to the prior Torelli result in \cref{relation to prior torelli}. In algebraic geometry, our result offers a new tool for examining the fibers of the compactified Torelli map from $\overline{M}_g$ to $\overline{A}_g$.

In \cref{history}, we discuss the literature informing this result, and the many reasons the objects under examination are considered of interest. \Cref{technical background} begins the dive into the definitions and constructions we will use. In \cref{oriented functoriality} we describe how certain constructions in the background of many papers can be made functorial. Stated in this context, Whitney's $2$-isomorphism theorem raises a question which may be of interest to combinatorialists.

\Cref{preservation of the theta divisor} is occupied with proving the main theorem. We begin in \cref{preliminary constructions} by relating known graph invariants in a new way, and introduce \textit{rigidity}, a term for when these relationships are in the categorical sense natural. The key result in this subsection is that rigidity can be measured in terms of graph divisors. \Cref{Preservation of the First Abel-Jacobi Image} proves a technical lemma using standard arguments. In \cref{proof of rigidity equivalences} we show that rigidity can also be characterized in terms of images of Abel-Jacobi maps. In particular, rigidity is equivalent to preservation of the discrete theta divisor. In \cref{Geometric interpretation and the main result}, we interpret rigidity in terms of the structure of graphs, thereby proving the main result, \cref{geometric rigidity}. While we refer the reader to \cref{geometric rigidity} itself for a more general and precise claim, as an immediate consequence of our results we have the following.

\begin{theorem*} Let $G$ and $H$ be 2-edge connected graphs of genus at least two. Then their Jacobians admit a discrete theta divisor preserving isomorphism (equivalently, a rigid isomorphism) if and only if $G$ and $H$ are isomorphic.
\end{theorem*}

In \cref{applications examples and relations to other results} we provide examples and applications. Our examples lie in \cref{examples}. We circle back to Whitney's $2$-isomorphism theorem in \cref{relation to whitney 2 iso theorem} and explicitly strengthen the prior graph Torelli theorem in \cref{relation to cap torelli}. In \cref{as matroid lifting} we briefly discuss how our result gives a procedure for lifting isomorphisms of matroids to isomorphisms of graphs. Finally, in \cref{fibers of the compactified Torelli map} we note that the objects we are considering are pertinent to Caporaso and Viviani's work in \cite{caporaso_torelli_stable_2010} examining the fibers of the compactified Torelli map.

We conclude in \cref{software} by directing the reader to the author's development of relevant algorithms and data structures in SageMath. The author has implemented much of the computational content of this paper, as well as several algorithms prominent in \cite{backman_riemann-roch_2017}.
    
\section{History and Motivation} \label{history}

Every connected finite graph $G$ has an associated $\mathbb Z$-graded finite commutative group, which we will call $\Pic_{\dee}(G)$. Its degree zero subgroup $\Pic^0_{\dee}(G)$ has been repeatedly rediscovered or reframed by researchers in a wide variety of fields. Dynamicists considered it and related generalizations while investigating a phenomenon known as self-organized criticality. In this guise it came to be known as the \textit{Abelian sandpile model}. Combinatorialists considered it in the context of \textit{chip firing games} or the \textit{dollar game} or \textit{critical groups}. It turns out that $\Pic^0_{\dee}(G)$ knows surprising and interesting things about $G$ (for example, its order is the number of spanning trees of $G$).

Graph theory often arises in algebraic geometry when one takes degenerate limits of families of varieties: which and how many times irreducible components of the degeneration intersect each other can be encoded as a graph. Graph invariants can then be used to analyze possible degenerations. The Picard group of a graph famously arose in arithmetic geometry, through work of Raynaud and of Oda and Seshadri, who proved that the special fibers of N\'eron models of certain Jacobians have components indexed by $\Pic^0_{\dee}(G)$ for $G$ a related graph \cite{raynaud_specialisation_1970} \cite{oda_compactifications_1979}. See \cite{caporaso_recursive_2018} for more context. Connections between $\Pic_{\dee}(G)$ and algebraic geometry gained further attention due to work by Baker and Norine, who exposed surprising parallels between graphs and Riemann surfaces (for example, the existence of a Riemann-Roch theorem \cite{baker_riemann-roch_2007}).

On the other hand, an invariant which we call the graph Jacobian $J(G)$ arose in relation to the classical combinatorial theory of cuts and flows on graphs. A Riemann surface of genus $g$ has a Jacobian variety, a complex torus of complex dimension $g$ which can be interpreted as parameterizing holomorphic forms up to relations arising from the geometry of the surface. The graph Jacobian is a real flat torus which admits a similar interpretation, as described in \cite{bacher_lattice_1997} and \cite{baker_harmonic_2007}. As it turns out, $\Pic^0_{\dee}(G)$ embeds canonically into $J(G)$.

The surface and graph parallel is further strengthened by the existence of \textit{Abel-Jacobi maps} in both cases. Saving precise statements for later, given a base point (base vertex), the first Abel-Jacobi map tells us how to map the points (vertices) of our surface (graph) into the Jacobian (into $\Pic^0_{\dee}(G)$). Except in trivial cases, this first map is injective. The $k$-th Abel-Jacobi map takes $k$ points (vertices), maps them individually using the first Abel-Jacobi map, then sums them using the group operation. In the surface case, for $g \geq 2$ the $g-1$-th map has image a hypersurface. In the language of algebraic geometry this means that the image is a divisor (relative to the Jacobian). It is called the \textit{theta divisor}. In the graph case, we will be concerned with an analogue we call the discrete theta divisor, which is simply a subset of $\Pic^0_{\dee}(G)$, and therefore $J(G)$.

While the Jacobian is not a complete invariant for Riemann surfaces, a critical result called the Torelli theorem says, roughly, that a Riemann surface can be characterized up to isomorphism by its Jacobian together with its theta divisor. In its most classical form it concerns periods and intersection theory; in modern language it is about injectivity of a map between moduli spaces. The following form, which we have weakened for simplicity, is one which closely connects to our results.

\begin{theorem}
    (Weak classical Torelli theorem) Let $X$ and $X'$ be Riemann surfaces of genus at least two, with specified base points. Let $(J(X), \theta)$ and $(J(X'), \theta')$ be Jacobians together with the theta divisors for $X$ and $X'$. If $J(X)$ and $J(X')$ are isomorphic in a way which identifies $\theta$ and $\theta'$, then $X \cong X'$.
\end{theorem}

Caporaso-Viviani proved a Torelli style result for graphs in \cite{caporaso_torelli_2010}. Independently, Su-Wagner proved the same result in the more general context of regular matroids \cite{su_lattice_2010}. The graph Torelli theorem relates isomorphisms of Jacobians to isomorphisms of graphic matroids, and does not make use of any analogue of the theta divisor. This result has a key application in algebraic geometry, which we will discuss in \cref{fibers of the compactified Torelli map}. 

We would be remiss in not mentioning that much of the above can be discussed equally well in the tropical context. See \cite{baker_degeneration_2015} for a survey.

This paper largely consists of the refinement of data from discrete theta divisors of graphs. In order to hold this data, we use the \textit{generalized cycle-cocycle system} of a graph, defined in \cite{backman_riemann-roch_2017}, which consists of equivalence classes of partial orientations on the underlying graph. This invariant is explicitly intended to package data from $\Pic_{\dee}(G)$ in a more ``matroidal'' form. While it cannot describe all of $\Pic_{\dee}(G)$, it can speak fluently about divisors which lie in a certain range of degrees. These include the \textit{special divisors}, which are closely related to the discrete theta divisor. As the arguments in this paper show, the generalized cycle-cocycle system is extremely well behaved in the situation of the graph Torelli theorem, and so all our definitions are intended to load as much information into it as possible.

\section{Technical background and relations between graph invariants} \label{technical background}

In this section, we describe the relationships between various graph invariants we will need, and indicate our preferred terminology for them.

We will need to discuss both directed and undirected graphs. 

\begin{defn}
A \textit{graph} $G$ is a pair $(V(G), E(G))$ of edge and vertex sets, with each edge associated with an unordered pair of distinct adjacent vertices $\{x,y \}$. This disallows edges which go from a vertex to itself. If each pair of vertices has at most one edge between them, we say $G$ is \textit{simple}. At all times we will only consider connected graphs. The empty graph is connected. A graph is \textit{$n$-connected} (\textit{$n$-edge connected}) if removing $n-1$ vertices (edges) never disconnects the graph. It is $n$-\textit{connective} ($n$-\textit{edge connective}) if $n$ is the greatest number for which $G$ is $n$-connected ($n$-edge connected). A \textit{based} graph $(G, e)$ is a graph $G$ with a choice of a distinguished edge $e$.

An \textit{isomorphism} of graphs $\psi: G \to H$ is a pair of bijections $V(G) \to V(H)$ and $E(G) \to E(H)$, also denoted by $\psi$. We require that $\psi$ preserve all adjacency relations between edges and vertices. If $(G, e)$ and $(H, w)$ are based, we also require that $\psi$ preserve the base edge.

A \textit{path} in a graph is a sequence $e_1, \dots, e_n$ of edges such that $e_i$ and $e_{i+1}$ are adjacent for each $i$. A \textit{cycle} is a path such that $e_1$ and $e_n$ are also adjacent. The vertices which lie between each $e_i$ and $e_{i+1}$ are referred to as being visited by the path or cycle. We also consider the ends of $e_1$ and $e_n$ not visited in the subsequent and prior steps, respectively, to be visited. The order in which vertices are visited is uniquely defined except in the case of a path entirely between two vertices; this case should never be problematic. We will frequently stretch our notation by referring to the set of edges composing a cycle as a cycle. When the possibility of confusion exists we will refer to such sets as \textit{unordered cycles}. \end{defn}

It is difficult to keep track of vertices using the data of the objects we are interested in. A motif of this paper is that we prefer to track oriented edges, and specify vertices as being the targets and origins of edges. To this end, we now define directed graphs. We put hats over symbols as a purely informal method of suggesting that we consider them with orientation data: for example $\hat e$ with $e$ an edge suggests that $\hat e$ is relevantly oriented. Usually this means that its target will be at least implicitly relevant.

\begin{defn}A \textit{directed graph} is a quadruple $(V(G), E(G), o, t)$ with an origin or source function $o: E(G) \to V(G)$ and a target function $t: E(G) \to V(G)$, for which the underlying undirected graph satisfies the same conditions as before. We will call a choice of such an $o$ and $t$ an \textit{orientation} on the graph, say $\Gamma$, and denote an oriented graph $\hat G$ by $(G, \Gamma)$. We likewise write a based directed graph as $(G, \Gamma, e)$. We may omit $\Gamma$ or $e$ when they are irrelevant or clear from context. We will sometimes consider graphs in which possibly only some of the edges are oriented. We call such orientations \textit{partial orientations}. In an oriented or partially oriented graph, a vertex is a \textit{source} if it has no incoming edges.

An isomorphism of (based) directed graphs is an isomorphism of the underlying (based) graphs. No notion of compatibility in orientations is required.
\end{defn}

\begin{defn}
A path (cycle) is \textit{simple} if it does not visit any vertex twice, except for possibly beginning and ending at the same vertex. Equivalently, we cannot extract any subcycles.
\end{defn}

A simple unordered cycle has only two possible directions. When an orientation is present, we will often indicate a direction for a simple unordered cycle by specifying whether we travel with or against a particular edge.

Several invariants will require a choice of orientation in their definitions. They vary up to a canonical isomorphism based on this choice, a fact that will be of significance in \cref{preliminary constructions}. We begin with invariants described in \cite{bacher_lattice_1997} (for an expository discussion, see \cite{bollobas_modern_1998}). Consider the euclidean vector space of $1$-chains $C_1(G, \R)$ with orthonormal basis the edges of $G$. Its dual space $C^1(G, \R)$, endowed with the induced inner product, is the space of \textit{$1$-cochains} of $G$, which we always think of as having the canonical orthonormal dual basis given by $\{e^* \}_{e \in E(G)}$. 

We denote the span of the algebraic cycles of $\hat G$ by $H^1(G, \R)$, and the projection operator onto $H^1(G, \R)$ will always be denoted by $\pi$. The integer lattice $C^1(G, \Z)$ in $C^1(G, \R)$ restricts to $H^1(G, \R)$ to produce a lattice $H^1(G, \Z)$, which is itself generated by algebraic cycles. There is always at least one basis of $H^1(G, \Z)$ given by simple algebraic cycles. The orthogonal complement of $H^1(G, \R)$ is spanned by the consistently oriented cuts of $G$.

\begin{defn}
	Given a path (cycle) $P = e_1, \dots, e_n$ in $\hat G$, its associated \textit{algebraic path (cycle)}, denoted $\alpha(P)$, is the formal sum $\sum_i s_i e_i^* \in C^1(G, \Z)$, where $s_i = 1$ if $P$ crosses $e_i$ in the direction of its orientation, and $s_i = -1$ otherwise. When $C$ is a cycle $\alpha(C)$ is always in $H^1(G, \Z)$. We say that $P$ is an \textit{oriented path (cycle)} if all the $s_i$ are $1$. The traversal of a path (cycle) $P$ backwards is indicated by $P\op$. We therefore have $\alpha(P\op) = -\alpha(P)$.
\end{defn} 

The \textit{Jacobian} $J(G) := H^1(G, \R)/H^1(G, \Z)$ is a flat metric torus. The Jacobian has a dual, the \textit{Albanese} $A(G)$, which can also be defined through homological means by quotienting the space $H_1(G, \R) = H^1(G, \R)^* \subset C_1(G, \R)$. There is a canonical finite subgroup of $J(G)$, which we call the \textit{discrete Jacobian}, denoted by $J_{\dee}(G)$. This is the discriminant group of the lattice $H^1(G, \Z)$. In words, take the dual lattice $\Hom_\Ab(H^1(G, \Z), \Z)$, which via the inner product we think of as itself a lattice in $H^1(G, \R)$, and quotient it by $H^1(G, \Z)$. When $a, b \in H^1(G, \R)$, we write $a \equiv b$ to indicate equivalence modulo $H^1(G, \Z)$.

In \cite{baker_riemann-roch_2007} Baker and Norine recast another closely related graph invariant, one which has appeared in many guises, including in \cite{bacher_lattice_1997}. Consider the free commutative group $\Div G$ on the vertices of $G$, referred to as the group of \textit{divisors} on $G$. This group has a $\Z$ grading into degrees $\Div^n G$, given by summing coefficients. The graph Laplacian matrix $\Delta$ of $G$ is the difference of the adjacency matrix and the degree matrix. It can be regarded as an endomorphism of $\Div G$, mapping into $\Div^0 G$. Taking $\Div G / \im \Delta$ produces a commutative group $\Pic_{\dee}(G)$, the \textit{discrete Picard group} of $G$, which inherits a $\Z$-grading into parts $\Pic_{\dee}^n(G)$. Each $\Pic_{\dee}^n(G)$ is finite.  Two divisors $D$ and $D'$ which are equivalent modulo $\im \Delta$ are referred to as \textit{linearly equivalent}, and we write $D \sim D'$. A divisor $D$ with only non-negative coefficients is referred to as \textit{effective}. When we write an inequality between divisors, we mean it to hold coefficient by coefficient. For example, $D \geq 0$ indicates that $D$ is effective. It is sometimes convenient to denote the $v$ coefficient of $D$ by $D(v)$. Finally, we denote by $|D|$ the set of effective divisors linearly equivalent to $D$.

It turns out that there is a canonical isomorphism $\Pic_{\dee}^0(G) \cong J_{\dee}(G)$. We need to describe how it is produced.

\begin{defn} \label{path from basepoint notation}
For each $w \in E(\hat G)$, define $h_w \in H^1(G, \R)$ to be $\pi w^* \in H^1(G, \R)$. If $(\hat G, \hat e)$ is based, for any vertex $v$ we define $P_v \in J_{\dee}(G)$ to be $\pi \alpha(P) + H^1(G, \Z)$, where $P$ is a path from $t(\hat e)$ to $v$ (recall that $\pi$ is the projection operator $C^1(G, \R) \to H^1(G, \R)$). In words, $h_w$ is the projection of the indicator function for $w$ to $H^1(G, \R)$, while $P_v$ is obtained by projecting any algebraic path from $t(\hat e)$ to $v$ to $H^1(G, \R)$, then passing to $J_{\dee}(G)$. 
\end{defn}

We now define a map $\Div^0(G) \to J_{\dee}(G)$ by extending $t(w) - o(w) \mapsto h_w + H^1(G, \Z)$ linearly. One can demonstrate that this is well defined, surjective, and the kernel is exactly $\im \Delta$, yielding the desired isomorphism $\Pic_{\dee}^0(G) \cong J_{\dee}(G)$.

Let $V(G)^{(n)}$ denote the $n$-th symmetric product of $V(G)$. We have the additional data, depending on $\hat e$, of \textit{Abel-Jacobi maps} $S_{\hat e}^n: V(G)^{(n)} \to \Pic_{\dee}^0(G)$, given by $v_1 + v_2 + \dots + v_n \mapsto v_1 + v_2 + \dots + v_n - nt(\hat e)$. After composing with our isomorphism $\Pic_{\dee}^0(G) \cong J_{\dee}(G)$, this becomes $\sum_{i=1}^n v_i \mapsto \sum P_{v_i}$. Stretching our notation, we also denote this map by $S_{\hat e}^n$. 

Let $g(G)$ denote the \textit{genus} of $G$, which we define to be $|E(G)| - |V(G)| + 1$ (graph theorists sometimes call this the cyclotomic number of $G$, reserving ``genus'' for a different invariant).

\begin{defn}
We refer to $\im S_{\hat e}^{g(G)-1}$, in either $\Pic^{0}_{\dee}(G)$ or $J(G)$, as the \textit{discrete theta divisor} $\theta_{\hat e}$ of $(G, \hat e)$. We refer to the divisors of $\theta_{\hat e} + (g-1)t(\hat e) \subset \Pic^{g-1}_{\dee} (G)$ as the \textit{special divisors}, the set of which is denoted $\mc S(G)$. This is independent of the base edge and consists exactly of the divisors linearly equivalent to an effective divisor. The rest of the divisors in $\Pic^{g-1}_{\dee} (G)$ are referred to as \textit{nonspecial}. We denote the set of these by $\mc N(G)$. \end{defn}

\begin{remark}
    In the Riemann surface context, the theta divisor has essentially the same definition, but there it is a hypersurface within the Jacobian. In the language of algebraic geometry, such a hypersurface is a divisor, whence the divisor part of ``theta divisor.'' The reader should not be mislead into thinking a theta divisor is a divisor on a graph or Riemann surface. In particular, the discrete theta divisor carries no more structure than that of a (finite) subset of the Jacobian.
\end{remark}

\begin{defn} \label{matroid}
    Each graph $G$ has an associated \textit{graphic matroid} $M(G)$, which is the edge set of $G$ together with the set of all simple unordered cycles in $G$.
    
    An isomorphism between graphic matroids is a bijection between the edge sets which induces bijections between the associated sets of simple cycles. When we write $(M(G), e)$ for $e$ an edge in $G$, we mean $M(G)$ to have base element $e$. Isomorphisms between graphic matroids with base elements will always be required to respect the base elements.
\end{defn}

Independently, both Caporaso-Viviani and Su-Wagner related $A(G)$ to $M(G)$ by proving a Torelli theorem for graphs, from combinatorial and algebraic geometry points of view, respectively
\cite{su_lattice_2010}, \cite{caporaso_torelli_2010}. For the moment we can state the results as follows. 

\begin{theorem} \label{capTorelli}
Let $G, H$ be $2$-edge connected graphs. Then $M(G) \cong M(H)$ if and only if $A(G) \cong A(H)$.
\end{theorem}

Since $A(G) \cong A(H)$ if and only if $J(G) \cong J(H)$, we obtain the same result for Jacobians. In \cref{relation to cap torelli} we will go into more detail about how the indicated isomorphisms can be produced; this will be brief and the only other mention of $A(G)$.

There is one more invariant we need to understand, based on partial orientations of graphs. Several of the following definitions and results can be found in \cite{backman_riemann-roch_2017}. They are in different but easily translated notation in that source.

\begin{defn}
Let $G$ be a graph and $\mb O^k(G)$ be the set of partial orientations on $k + |V(G)|$ edges, when this is nonempty. When $k = g-1$, this is the set of full orientations and we just write $\mathbb O(G)$. We let $c: \mathbb O^k(G) \to \Pic^k_{\dee}(G)$ be the map which takes a partial orientation $U$ to the divisor $\sum_{e \in E(G)} t(e) - \sum_{v \in V(G)} v$, where each $t(e)$ is the target of an edge oriented according to $U$. We call this the \textit{Chern class map} and call $c(U)$ the \textit{Chern class} of $U$.
\end{defn}

It is natural to ask which divisors are given by applying the Chern class map to partial orientations. Such divisors are called \textit{partially orientable}. This question has an answer which is uniform within each Chern class.

\begin{theorem} \textit{(Backman)}
Let $D$ be a divisor of degree at most $g-1$. Then $D$ is partially orientable if and only if $|D + \sum_{v \in V(G)} v| \neq \emptyset$.
\end{theorem}

\begin{proof}
    \cite{backman_riemann-roch_2017}.
\end{proof}

We can put this together with the following well known consequence of Riemann-Roch.

\begin{lemma} \textit{(Baker-Norine)}
Let $D$ be a divisor of degree at least $g(G)$. Then $|D| \neq \emptyset$.
\end{lemma}
\begin{proof}
\cite{baker_riemann-roch_2007}.
\end{proof}

\begin{cor} \label{partially orientable}
Let $D$ be a divisor of degree at least $g(G) - |V(G)|$. Then $D$ is partially orientable.
\end{cor}

Quotienting by equality of Chern classes, we obtain $\overline{\mathbb O^k}(G)$ and $\overline{\mathbb O}(G)$. We also call the induced injections $c: \overline{\mathbb O^k}(G) \to \Pic^{k}_{\dee}(G)$ the \textit{Chern class maps}. 

\begin{defn}
    If $K$ is the set of partially orientable divisors of degree $k$, we denote the inverse of the Chern class map by $\OO: K \to \overline{\mb O^k}(G)$.
\end{defn}

Call partial orientations $U$ and $U'$ \textit{equivalent} if $[U] = [U']$ in $\overline{\mathbb O^k}(G)$.

\begin{theorem} \textit{(Backman)} \label{linear equivalence of orientations}
Two partial orientations are equivalent if and only if one can be transformed into the other by a sequence of the following operations. \begin{enumerate}
    \item Reverse all edges in a consistently oriented cycle.
    \item Reverse all edges in a consistently oriented cut.
    \item Take an edge $\ell$ pointing at a vertex $v$ and an unoriented edge $r$ adjacent to $v$. Then deorient $\ell$ and orient $r$ to point at $v$.
\end{enumerate}
\end{theorem}
\begin{proof}
    \cite{backman_riemann-roch_2017}.
\end{proof}

Sometimes we will fix sets of edges which we do not wish to be oriented in a partial orientation.

\begin{defn}
Let $X = \{e_1, e_2,\dots, e_r \} \subseteq E(G)$. Then we let $\mathbb O(G, X) \subseteq \mathbb O^{|E(G)| - r}(G)$ be the set of partial orientations on $G$ in which exactly the $e_i$ are unoriented. The quotient of this set by equivalence is denoted by $\overline{\mathbb O}(G, X)$.
\end{defn}

The orientation analogues of $\mathcal S(G)$ and $\mathcal N(G)$ will be of particular interest to us.

\begin{defn}
We define $\mc S_{\mathbb O}(G) \subseteq \overline{ \mb O}(G)$ to be the preimage of $\mc S(G)$ under the Chern class map, and likewise for $\mc N_{\mathbb O}(G)$ and $\mc N(G)$.
\end{defn}

The two sets just defined have a characterization which will be essential to our later reasoning. An orientation class is in $\mc S_{\mathbb O}(G)$ exactly when it contains a sourceless orientation, while the representatives of classes in $\mc N_{\mathbb O}(G)$ consist of exactly the acyclic orientations. More generally, we have the following result.

\begin{theorem} \textit{(Backman)} \label{geometric effectiveness}
The partial orientations which are equivalent to sourceless partial orientations are exactly those with an effective divisor in their Chern class, while partial orientations equivalent to acyclic partial orientations are exactly those without an effective divisor in their Chern class.
\end{theorem}
\begin{proof}
    \cite[Theorem 4.10]{backman_riemann-roch_2017}.
\end{proof}

A classical theorem of Whitney says that all isomorphisms of graphic matroids can be produced through sequences of certain topological operations on their associated graphs. We conclude this section by couching this theorem in our language.

\begin{defn}
A \textit{cyclic bijection} between graphs consists of a bijection between their edges which preserves simple unordered cycles. This is precisely the same as an isomorphism of graphic matroids as in \cref{matroid}. Let $\Cyc$ be the category of $2$-connected based graphs and base preserving cyclic bijections between them. Within the morphisms of $\Cyc$ there is a distinguished class of \textit{edge isomorphisms}, which are those which extend to a genuine isomorphism of graphs.
\end{defn}

\begin{example} \label{not edge isomorphisms}
Consider the following objects in $\Cyc$.
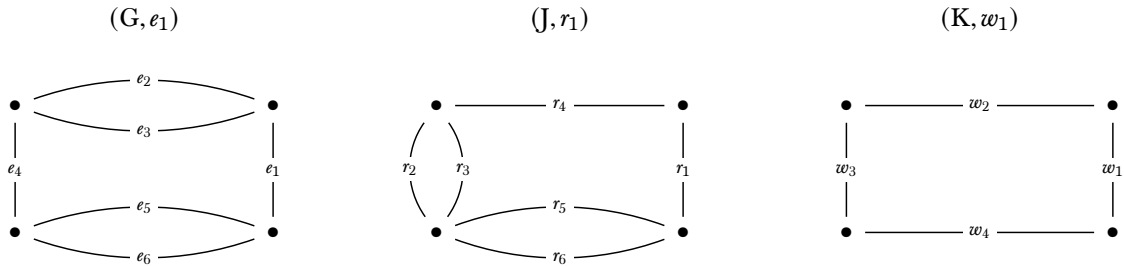
\[\begin{tikzcd}
	& {(G, e_1)} &&&& {(J, r_1)} &&&& {(K, w_1)} \\
	\bullet && \bullet && \bullet && \bullet && \bullet && \bullet \\
	\\
	\bullet && \bullet && \bullet && \bullet && \bullet && \bullet
	\arrow["{r_6}"{description}, curve={height=-12pt}, no head, from=4-7, to=4-5]
	\arrow["{r_5}"{description}, curve={height=-12pt}, no head, from=4-5, to=4-7]
	\arrow["{r_2}"{description}, curve={height=12pt}, no head, from=2-5, to=4-5]
	\arrow["{r_3}"{description}, curve={height=12pt}, no head, from=4-5, to=2-5]
	\arrow["{r_4}"{description}, no head, from=2-7, to=2-5]
	\arrow["{r_1}"{description}, no head, from=4-7, to=2-7]
	\arrow["{e_1}"{description}, no head, from=4-3, to=2-3]
	\arrow["{e_5}"{description}, curve={height=12pt}, no head, from=4-3, to=4-1]
	\arrow["{e_6}"{description}, curve={height=-12pt}, no head, from=4-3, to=4-1]
	\arrow["{e_2}"{description}, curve={height=12pt}, no head, from=2-3, to=2-1]
	\arrow["{e_3}"{description}, curve={height=-12pt}, no head, from=2-3, to=2-1]
	\arrow["{e_4}"{description}, no head, from=2-1, to=4-1]
	\arrow["{w_2}"{description}, no head, from=2-9, to=2-11]
	\arrow["{w_1}"{description}, no head, from=2-11, to=4-11]
	\arrow["{w_3}"{description}, no head, from=4-9, to=2-9]
	\arrow["{w_4}"{description}, no head, from=4-9, to=4-11]
\end{tikzcd}\]
There is a morphism from $(G, e_1)$ to $(J, r_1)$ given by $e_i \mapsto r_i$. This cannot be an edge isomorphism, because the two underlying graphs are not isomorphic. However, a cyclic bijection can fail to be an edge isomorphism more subtly. Any permutation of the $w_i$ which fixes $w_1$ yields an endomorphism of $(K, w_1)$. Take $\sigma$ to be the one which swaps $w_2$ and $w_3$, and fixes $w_4$. This is also not an edge isomorphism, for if we were to extend $\sigma$ to a graph isomorphism it would not preserve the adjacency of $w_1$ and $w_2$.
\end{example}

In the case that a graph has only two vertices, edge isomorphisms do not extend uniquely to graph isomorphisms, but in all other cases such an extension is unique.

\begin{defn}
Let $G$ be a $2$-connected graph with a nonempty subgraph $X$. Suppose that there is a (necessarily unique) nonempty subgraph $X^*$ of $G$ with the properties that $G = X \cup X^*$ and that $X \cap X^*$ consists of exactly two vertices, say $v$ and $w$. Further suppose that each of $X$ and $X^*$ contains at least three vertices. Then we call $X$ an \textit{arch} of $G$, call $v$ and $w$ the \textit{tips} of $X$, and in a stretch of notation call $X^*$ the \textit{complement} of $X$.
\end{defn}

\begin{defn}
Let $X$ be an arch of the based graph $(G, e)$ which does not contain $e$. The \textit{Whitney move} on $X$ consists of taking $X \sqcup X^*$ and gluing the copy of $v$ in $X$ over the copy of $w$ in $X^*$, and the copy of $w$ in $X$ over the copy of $v$ in $X^*$. The graph thus produced is called $\mc W_X G$. We give it the base edge $e$.
\end{defn}

When $G$ is $2$-connected, the identity yields a morphism $(G, e) \to (\mc W_X G, e)$ in $\Cyc$, which by a stretch of notation we also refer to as a Whitney move and denote by $\mc W_X$. By construction $M(-)$ can be regarded as a functor from $\Cyc$ to the category of based graphic matroids, in which the arrows are isomorphisms - for this reason, we denote the isomorphism of matroids associated with a cyclic bijection $\vphi$ by $M(\vphi)$. Note that $M(G) = M(\mc W_X G)$, and in fact $M(\mc W_X) = \id_{M(G)}$.

\begin{example}
    The morphism from $(G, e_1)$ to $(J, r_1)$ of \cref{not edge isomorphisms} is, after relabeling edges, an example of a Whitney move. The endomorphism of $(K, w_1)$ from the same example is also a Whitney move.
\end{example}

\begin{theorem}[Whitney's 2-isomorphism theorem] \label{whitney 2 iso}
Let $\phi: (G, e) \to (H, w)$ be a morphism in $\Cyc$. Then $\phi$ can be written as a composition $\tau \circ \mc W_{X_n} \circ \dots \circ \mc W_{X_1}$, where $\tau$ is an edge isomorphism with $M(\phi) = M(\tau)$.
\end{theorem}

\begin{proof}
See \cite{truemper_whitneys_1980} for a pleasant and simple proof.
\end{proof}

By allowing a second type of move it is easy to extend this result to all connected graphs, but that case is not of interest to us. Notice that since both the arch and its complement have at least three vertices, removing the tips of an arch always disconnects a graph. Otherwise, the Whitney move associated with this arch is an edge isomorphism. Thus this theorem implies that any cyclic bijection between $3$-connected graphs is actually an edge isomorphism, since the only possible Whitney moves for these graphs are edge isomorphisms.

\section{Oriented Functoriality} \label{oriented functoriality}
We now wish to describe the Jacobian and our other, related objects in functorial terms. Ideally, the domain for our functors would be $\Cyc$, but the objects of this category carry too little data to support more than the previously discussed functor to graphic matroids. To capture our entire range of discussion, the minimal extra data necessary is that of $\Cyc$ but with each graph also being directed.

\begin{defn}
Let $\orCyc$ be the category of $2$-connected and $2$-edge connected based oriented graphs and base preserving cyclic bijections between them.
\end{defn}

\begin{example}
    The morphisms between objects in $\orCyc$ see nothing of the orientation data. Put another way, the forgetful functor $\orCyc \to \Cyc$ which discards the base orientations is full and faithful. Thus we could orient the edges in \cref{not edge isomorphisms} arbitrarily and obtain morphisms in $\orCyc$. The difference is that having specified base orientations as well as base edges allows us to do computations using these morphisms in a canonical way. \Cref{functoriality} describes this.
\end{example}

\begin{theorem} \label{functoriality}
Let $\vphi: (\hat G, \hat e) \to (\hat H, \hat w)$ be a morphism in $\orCyc$. For each $\ell \in E(G)$, choose a sign $\sgn_\vphi(\ell) \in \{-1, 1\}$, and let $\ell \mapsto \sgn_\vphi(\ell) \vphi(\ell)^*$ extend linearly to an inner product preserving isomorphism $\vphi_*: C^1(G, \R) \to C^1(H, \R)$. There is a unique choice of signs so that $\sgn_\vphi(e) = 1$ and $\vphi_*$ restricts to an isomorphism between $H^1(G, \Z)$ and $H^1(H, \Z)$ (and therefore $H^1(G, \R)$ and $H^1(H, \R)$).
\end{theorem}
\begin{proof}
Choose an edge $u \neq e$ in $G$. By Menger's theorem \cite{bollobas_modern_1998}, there is a simple cycle $C = \ell_1, \dots, \ell_m$ which crosses both $e$ and $u$, say with $\ell_1 = e$ and in the direction aligned with $e$. Let $W$ be the cycle in $H$ given by the unordered simple cycle $\vphi(C)$, in the direction matching $w$.

Let $\alpha(C) = \sum_i s_i \ell_i^*$ and $\alpha(W) = \sum_{i} q_i \vphi(\ell_i)^*$. Set 
\begin{equation*} \sgn_\vphi(\ell_i) = q_i/s_i\end{equation*}
for each $\ell_i$ appearing in $C$. This choice of signs is engineered so that $\vphi_*$ will take $\alpha(C)$ to $\alpha(W)$.

We need to show that $\vphi_*$ is well defined, and in particular that our assigned values of $\sgn_\vphi(\ell_i)$ do not depend on our choice of $C$. Let $C' = r_1, \dots, r_k$ be another simple cycle in $G$ containing $e$ and proceeding in the direction matching $e$, say with $r_1 = e$. Suppose there is a $u = \ell_i = r_j$ other than $e$ which $C$ and $C'$ cross in the same direction. Because $C$ and $C'$ are simple and share at least two edges, each must visit at least three vertices.

Let $W'$ be the cycle in $H$ given by the unordered simple cycle $\vphi(C')$, in the direction matching $w$. Write $\alpha(C') = \sum_i s'_i r_i^*$ and $\alpha(W') = \sum_{i} q'_i \vphi(r_i)^*$. What we need is that $q_i/s_i = q'_j/s_j'$. Note that $u$ is crossed by both $C$ and $C'$ in the same direction exactly when $s_i = s_j'$. Similarly, $W$ and $W'$ cross $\vphi(u)$ in the same direction exactly when $q_i = q_j'$. Thus $q_i/s_i = q'_j/s_j'$ is equivalent to the claim that $C$ and $C'$ cross $u$ in the same direction if and only if $W$ and $W'$ cross $\vphi(u)$ in the same direction. In fact the forward implication will be sufficient, since we can then apply the same reasoning to $\vphi\inv: (H, w) \to (G, e)$. This means it will be sufficient to show that $W$ and $W'$ visit the same vertex just before crossing $\vphi(u)$.

We know from \cref{whitney 2 iso} that as a morphism of objects in $\Cyc$, we can decompose $\vphi$ into a sequence of Whitney moves followed by an edge isomorphism. An edge isomorphism extends to a graph isomorphism which preserves the sequence of vertices a cycle visits, provided the cycle visits at least three vertices. It follows that it is enough to show that if $\vphi$, as a morphism in $\Cyc$, is given by a single Whitney move, then $W$ and $W'$ cross $\vphi(u)$ in the same direction.

Let $v_1, v_2$ be the tips of an arch $X$ in $G$. Let $a_1, \dots, a_f$ be the sequence of vertices in $X$ which $C$ visits, and likewise for $b_1, \dots, b_g$ and $C'$. After the Whitney move on $X$, we have that $W$ visits $a_f, a_{f-1}, \dots, a_1$ in $X$ and $W'$ visits $b_{g}, b_{g-1}, \dots, b_1$ in $X$, and the sequence of vertices visited is otherwise the same. It follows that both $W$ and $W'$ cross $\vphi(u)$ in the same direction. This shows that $\vphi_*$ is well defined.

By construction $\vphi_*$ takes any simple algebraic cycle $\alpha(C)$ into $H^1(H, \Z)$, provided that $C$ contains $e$. Let $\alpha(U)$ be a simple algebraic cycle in $G$ not containing $e$. If $U$ does not visit both end points of $e$, let $Y$ be a simple cycle containing both $e$ and some edge in $U$. If $U$ does visit both end points of $e$, then $U$ contains (exactly one) edge $r$ which is parallel to $e$, and we let $Y$ be $U$ but with $r$ replaced with $e$. 

Let $v_1$ and $v_2$ be the first and last, respectively, vertices visited by $Y$ that $U$ also visits. Following $Y$ to $v_1$, following $U$ in the two possible directions to $v_2$, and then following $Y$ back to $e$ gives two simple cycles $Y_0$ and $Y_1$ with $\alpha(Y_0) - \alpha(Y_1) = \pm \alpha(U)$. Thus $\vphi_*(\alpha(Y_0)) - \vphi_*(\alpha(Y_1)) = \pm (\vphi_*(\alpha(U))) \in H^1(H, \Z)$.

Since simple cycles generate $H^1(G, \Z)$, it follows that $\vphi_*$ takes $H^1(G, \Z)$ into $H^1(H, \Z)$. From the definitions we have that $(\vphi\inv)_* = \vphi_*\inv$, so $\vphi_*$ actually restricts to an isomorphism $H^1(G, \Z) \cong H^1(H, \Z)$.

It only remains to check uniqueness. Let $T$ be given by any choice of signs with $T(e^*) = w^*$, and suppose that $T$ restricts to an isomorphism $H^1(G, \Z) \cong H^1(H, \Z)$. For any $u \in E(G)$ other than $e$, choose a simple cycle $C$ in $G$ which contains $e$ and $u$. Then $\vphi_*(\alpha(C)) - T(\alpha(C))$ has only even coefficients, and has $w^*$ coefficient zero. Let $t$ be its $\vphi(u)^*$ coefficient. Let $K \subset V(H)$ be $t(w)$ and all further vertices that $\vphi(C)$ visits before crossing $\vphi(u)$. The set of edges $K'$ between vertices in $K$ and vertices not in $K$ is a minimal cut, and $K' \cap \vphi(C) = \{w, \vphi(u)\}$. For each $\ell \in K'$, let $a_\ell = 1$ if $\ell$ points at $K$ and $a_\ell = -1$ otherwise. Then $\sum_{\ell \in K'} a_\ell \ell^*$ is orthogonal to $H^1(H, \R)$, and in particular \begin{align*}0 &= \left \langle \sum_{\ell \in K'} a_\ell \ell^*, \vphi_*(\alpha(C)) - T(\alpha(C)) \right \rangle \\
	&= a_{\vphi(u)} t \end{align*}
Therefore $t = 0$. Since $u$ was arbitrary it follows that $\vphi_* = T$.
\end{proof}

\begin{remark}
If we allowed $1$-connective or $1$-edge connective graphs, the uniqueness in the prior result would always fail for them.
\end{remark}

\begin{defn}
Given a $\vphi$ as in the prior theorem, let $\vphi_*$ and $\sgn_\vphi: E(G) \to \{-1, 1\}$ be the unique functions specified by the theorem. We say that $\vphi$ and $\vphi_*$ are \textit{orientation preserving} if every $\sgn_\vphi(\ell) = 1$, and that $\vphi$ \textit{reverses} $\ell$ if $\sgn_\vphi(\ell) = -1$.
\end{defn}

\begin{example} \label{computation of pushforward} Consider the following objects in \orCyc.

\[\begin{tikzcd}
	& \bullet && \bullet && \bullet && \bullet \\
	{(\hat G, \Gamma_1, \hat e_1)} &&&&&&&& {(\hat H, \Gamma_2, \hat r_1)} \\
	& \bullet && \bullet && \bullet && \bullet \\
	&& \bullet &&&& \bullet
	\arrow["{e_3}"{description}, from=1-4, to=1-2]
	\arrow["{e_4}"{description}, curve={height=12pt}, from=1-2, to=3-2]
	\arrow["{e_5}"{description}, curve={height=-12pt}, from=1-2, to=3-2]
	\arrow["{e_1}"{description}, curve={height=-12pt}, from=3-4, to=1-4]
	\arrow["{e_2}"{description}, curve={height=-12pt}, from=1-4, to=3-4]
	\arrow["{e_7}"{description}, from=3-4, to=4-3]
	\arrow["{e_6}"{description}, from=3-2, to=4-3]
	\arrow["{r_7}"{description}, from=1-8, to=1-6]
	\arrow["{r_2}"{description}, curve={height=-12pt}, from=1-6, to=3-6]
	\arrow["{r_6}"{description}, from=3-6, to=4-7]
	\arrow["{r_3}"{description}, from=4-7, to=3-8]
	\arrow["{r_4}"{description}, curve={height=-12pt}, from=3-8, to=1-8]
	\arrow["{r_5}"{description}, curve={height=-12pt}, from=1-8, to=3-8]
	\arrow["{r_1}"{description}, curve={height=-12pt}, from=3-6, to=1-6]
\end{tikzcd}\]

There is a cyclic bijection $\vphi$ with $e_i \mapsto r_i$. We will compute $\vphi_*$ for this map. We pick out the following simple cycles in $G$. \begin{align*}
    C_1 &= e_1, e_3, e_4, e_6, e_7 & \alpha(C_1) &= e_1^* + e_3^* + e_4^* + e_6^* - e_7^* \\
    C_2 &= e_1, e_2 & \alpha(C_2) &= e_1^* + e_2^*\\
    C_3 &= e_4, e_5 & \alpha(C_3) &= e_4^* - e_5^*
\end{align*}
Taking $\alpha(C_1), \alpha(C_2), \alpha(C_3)$ yields a basis for $H^1(G, \Z)$ with the orientation $\Gamma_1$.

To see which edges $\vphi$ reverses, we take, for example, the simple cycle $C_1$ and determine coefficients so that \begin{align*}
	\vphi_*(\alpha(C_1)) &= \vphi_*(e_1^* + e_3^* + e_4^* + e_6^* - e_7^*) \\
	&= \sgn_\vphi(e_1)\vphi(e_1)^* + \sgn_\vphi(e_3)\vphi(e_3)^* + \dots + \sgn_\vphi(e_6)\vphi(e_6)^* - \sgn_\vphi(e_7)\vphi(e_7)^* \\
	&= r_1^* + \sgn_\vphi(e_3)r_3^* + \dots + \sgn_\vphi(e_6)r_6^* - \sgn_\vphi(e_7)r_7^*
\end{align*} 
is an algebraic cycle in $(\hat H, \Gamma_2)$. This can be done by following $\vphi(C_1)$ in the direction given by $r_1$ and noting whether we go with or against the direction of each $r_i$. We proceed similarly with $C_2$. 

On the other hand $C_3$ does not cross $e_1$, so we cannot compute the signs for it in the same way. We could proceed as in the proof of \cref{functoriality}, but a more efficient method is to observe that we have already determined that $\sgn_\vphi(e_4) = -1$. We can determine the unknown coefficient satisfying \begin{align*} \vphi_*(\alpha(C_3)) &= \vphi_*(e_4^* - e_5^*) \\
	&= -r_4^* - \sgn_{\vphi}(e_5)r_5^*
\end{align*} 
by following $\vphi(C_3)$ in the direction which goes against $r_4$, and observing that we cross $r_5$ against its direction. If $\alpha(\vphi(C_3))$ is to be an algebraic cycle, traveling against an edge's direction must give us a negative coefficient in $-r_4^* - \sgn_{\vphi}(e_5)r_5^*$, so that we must have $\sgn_{\vphi}(e_5) = 1$.

These computations show that $\vphi$ is given on the canonical basis of $C^1(G, \R)$, ordered by the subscripts of the $e_i$, by the matrix \[ \begin{bmatrix}
1 & 0 & 0 & 0 & 0 & 0 & 0 \\
0 & 1 & 0 & 0 & 0 & 0 & 0 \\
0 & 0 & -1 & 0 & 0 & 0 & 0 \\
0 & 0 & 0 & -1 & 0 & 0 & 0 \\
0 & 0 & 0 & 0 & 1 & 0 & 0 \\
0 & 0 & 0 & 0 & 0 & -1 & 0 \\
0 & 0 & 0 & 0 & 0 & 0 & 1 
\end{bmatrix}  \]
The nonzero entries are the values of $\sgn_\vphi$ on the corresponding edges. \end{example}

\begin{remark} 
\Cref{functoriality} implies that the invariants $C^1(-, \R), H^1(-, \R), H^1(-, \Z)$ and $J(-)$ are functorial with respect to $\orCyc$. When no confusion will result, we use $(-)_*$ for whichever of them is relevant. In particular, the induced maps $C^1(G, \R) \to C^1(H, \R)$ are given by signed permutation matrices, in which the signs of the entries indicates which edges $\vphi$ reverses. This implies that given a composition $\rho \circ \vphi$ in $\orCyc$, for any edge $e$ we have $\sgn_{\rho \circ \vphi}(e) = \sgn_\rho(\vphi(e))\sgn_\vphi(e)$. 
\end{remark}

Now we would like to specify an orientation for $\mc W_X \hat G$. There is a natural choice.

\begin{defn}
Let $\hat G$ be a directed graph with an arch $X$. We define an orientation on $\mc W_X \hat G$ as follows: let each edge in $X \sqcup X^*$ be oriented as in $\hat G$, and then allow this orientation to descend to the gluing of $X$ and $X^*$.
\end{defn}

\begin{prop}
For any $(\hat G, \hat e)$ and arch $X \subset E(G)$, we have that $\mc W_X$ reverses exactly the edges in $X$.
\end{prop}
\begin{proof}
Let $T$ be the set of edges $w$ such that such that either $w \in X$ and $\mc W_X$ reverses $w$, or $w \not \in X$ and such that $\mc W_X$ does not reverse $w$. By definition we have that $\hat e \in T$. Suppose $\hat a \in T$, $\hat b \in E(\hat G)$, and that $\hat a$ is adjacent to $v,w$ and $\hat b$ is adjacent to $w, x$ in $G$. Choose a simple path $L$ in $\hat G$ from $x$ to $v$ avoiding $w$, so that $\hat a, \hat b, L$ is a simple cycle. Then following the edges of $\{\hat a, \hat b, L\}$ starting with crossing $\hat a$ in the direction it points also yields a simple cycle in $\mc W_X \hat G$, except that passing from $X$ to $X\op$ or vice versa results in crossing edges in the opposite direction. Thus $\hat b \in T$. By connectedness $T = E(G)$.
\end{proof}

\begin{cor}
A composition of Whitney moves $\mc W_{X_n} \circ \dots \circ \mc W_{X_1}$ is orientation preserving if and only if every edge appears in an even number of $X_i$.
\end{cor}

An oriented version of Whitney's 2-isomorphism theorem is now immediate.

\begin{theorem}
Let $\phi: (\hat G, e) \to (\hat H, w)$ be a morphism in $\orCyc$. Then $\phi$ can be written as a composition $\phi = \tau \circ \mc W_{X_n} \circ \dots \circ \mc W_{X_1}$, where $\tau$ is an edge isomorphism with $M(\phi) = M(\tau)$. If $\tau$ is orientation preserving, then $\phi$ is orientation preserving if and only if each edge appears in an even number of the $X_i$. 
\end{theorem}

\begin{cor}
In the situation of the above theorem, let $r \in \hat G$ be any edge, and let $\tau$ be orientation preserving. Suppose we have any other factorization $\phi = \rho \circ \mc W_{Y_m} \circ \dots \circ \mc W_{Y_1}$ with $\rho$ orientation preserving. Let $a$ be the number of times $r$ appears in the $X_i$, and likewise for $b$ and the $Y_i$. Then $a \equiv b \mod 2$.
\end{cor}

Thus whenever such an orientation preserving $\tau$ exists, Whitney's theorem has a kind of canonicity modulo $2$. This raises a natural question.

\begin{question} \label{is orientation preserving possible}
	Let $\phi: (\hat G, e) \to (\hat H, w)$ be a morphism in $\orCyc$. Then $\phi$ can be written as a composition $\phi = \tau \circ \mc W_{X_n} \circ \dots \circ \mc W_{X_1}$, where $\tau$ is an edge isomorphism with $M(\phi) = M(\tau)$. Can directed graph invariants tell us whether we can choose $\tau$ to be orientation preserving?
\end{question}

\begin{example}
	Consider the following objects in $\orCyc$.

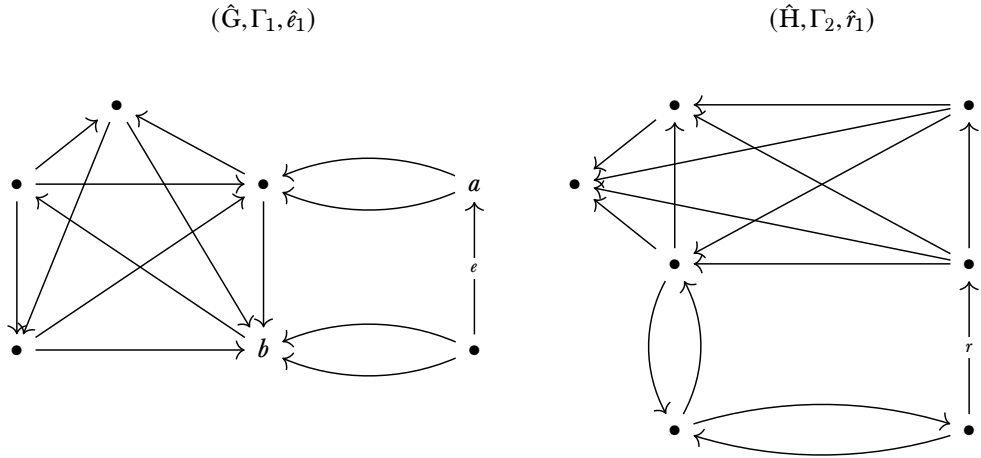
\[\begin{tikzcd}
	&& {(\hat G, \Gamma_1, \hat e_1)} &&&&& {(\hat H, \Gamma_2, \hat r_1)} \\
	& \bullet &&&&& \bullet && \bullet \\
	\bullet && \bullet && a & \bullet \\
	&&&&&& \bullet && \bullet \\
	\bullet && b && \bullet \\
	&&&&&& \bullet && \bullet
	\arrow[curve={height=-12pt}, from=6-9, to=6-7]
	\arrow[curve={height=-12pt}, from=6-7, to=6-9]
	\arrow[curve={height=12pt}, from=4-7, to=6-7]
	\arrow[curve={height=12pt}, from=6-7, to=4-7]
	\arrow[from=4-9, to=4-7]
	\arrow["r"{description}, from=6-9, to=4-9]
	\arrow["e"{description}, from=5-5, to=3-5]
	\arrow[curve={height=12pt}, from=5-5, to=5-3]
	\arrow[curve={height=-12pt}, from=5-5, to=5-3]
	\arrow[curve={height=12pt}, from=3-5, to=3-3]
	\arrow[curve={height=-12pt}, from=3-5, to=3-3]
	\arrow[from=3-3, to=5-3]
	\arrow[from=3-3, to=2-2]
	\arrow[from=3-1, to=2-2]
	\arrow[from=5-1, to=5-3]
	\arrow[from=3-1, to=3-3]
	\arrow[from=2-2, to=5-1]
	\arrow[from=2-2, to=5-3]
	\arrow[from=5-3, to=3-1]
	\arrow[from=5-1, to=3-3]
	\arrow[from=3-1, to=5-1]
	\arrow[from=4-9, to=2-9]
	\arrow[from=4-7, to=2-7]
	\arrow[from=2-7, to=3-6]
	\arrow[from=4-7, to=3-6]
	\arrow[from=2-9, to=2-7]
	\arrow[from=4-9, to=2-7]
	\arrow[from=2-9, to=4-7]
	\arrow[from=2-9, to=3-6]
	\arrow[from=4-9, to=3-6]
\end{tikzcd}\]
These admit a cyclic bijection (perform a Whitney twist on an arch with tips $a$ and $b$), but we cannot choose the $\tau$ of the question to be orientation preserving. To see this, observe that every edge in the copy of $K_5$ within $G$ is in an oriented cycle, but the copy of $K_5$ in $H$ is acyclic. For such a $\tau$ to exist would imply that after a sequence of Whitney moves on $G$, a cycle could be introduced to the copy of $K_5$. This is impossible, for any arc either includes all or none of the edges in any subgraph which is a copy of $K_5$.
\end{example}

\section{Preservation of the discrete theta divisor} \label{preservation of the theta divisor}

In this section, we begin by describing \textit{rigidity}, a property of morphisms in $\orCyc$. Our goal is \cref{extended torelli}, which provides several equivalent characterizations of rigidity. In particular, it says that rigidity is equivalent to both preservation of the discrete theta divisor and preservation of the first Abel-Jacobi image. This latter criterion will eventually put us in a position to describe the consequences of rigidity in terms of graph structures.

The examples of \cref{examples} are intended to be companions to this section and \cref{Geometric interpretation and the main result}. The reader is encouraged to refer to them early and often.

\subsection{Preliminary Constructions and the First Characterization of Rigidity} \label{preliminary constructions}

In order to define rigidity, we first need to describe an orientation analogue of $\vphi_*$.

\begin{defn} \label{dfn: signs of morphisms}
Given a morphism $\vphi: (\hat G, \Gamma, \hat e) \to (\hat H, \Gamma', \hat w)$ in $\orCyc$, let $\vphi_{\mathbb O}: \mathbb O(G) \to \mathbb O(H)$ be defined as follows. Given an orientation $U \in \mathbb O(G)$ and $r \in E(G)$, let $\sgn_U(r) = 1$ if $r$ is oriented the same way in $U$ and $\Gamma$, and let $\sgn_U(r) = -1$ otherwise. Note that with this notation, a simple cycle $C$ in $\hat G$, say $e_1, \dots, e_n$, is an oriented cycle in $U$ if and only if $\alpha(C) = \sum_i \sgn_U(e_i) h_{e_i}$.  

Let $\vphi_{\mathbb O}(U)$ be the orientation on $H$ where for each $r \in E(G)$, we have $\sgn_{\vphi_{\mathbb O}(U)}(r) = \sgn_\vphi(r)\sgn_U(r)$. This construction is functorial on $\orCyc$. We will show in \cref{well defined orientation pushforward} that $\vphi_{\mathbb O}$ always descends to $\overline{\mathbb O}(G) \to \overline{\mathbb O}(H)$, and we will also denote this function by $\vphi_{\mathbb O}$.

Finally, if $U$ is a partial orientation, we extend the above definitions by setting $\sgn_U(r) = 0$ exactly when $r$ is unoriented. In this way $\vphi_{\mathbb O}$ also extends to partial orientations.
\end{defn}

Algebraically, this is straightforward to compute once $\vphi_*$ is known. If we regard $U$ as being encoded by vector $v$ with entries $\pm 1$ in $C^1(G, \R)$, then $\vphi_{\mb O}(U)$ is the orientation encoded by $\vphi_*(v)$.

\begin{example} \label{computation of pushforward extended}
	We will show how $\vphi_{\mb O}$ works in the case of \cref{computation of pushforward}. Recall that we had a cyclic bijection defined by $e_i \mapsto r_i$, with the following objects.
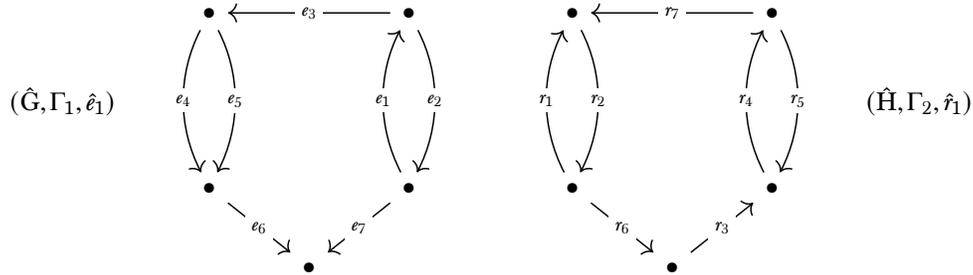
\[\begin{tikzcd}
	& \bullet && \bullet && \bullet && \bullet \\
	{(\hat G, \Gamma_1, \hat e_1)} &&&&&&&& {(\hat H, \Gamma_2, \hat r_1)} \\
	& \bullet && \bullet && \bullet && \bullet \\
	&& \bullet &&&& \bullet
	\arrow["{e_3}"{description}, from=1-4, to=1-2]
	\arrow["{e_4}"{description}, curve={height=12pt}, from=1-2, to=3-2]
	\arrow["{e_5}"{description}, curve={height=-12pt}, from=1-2, to=3-2]
	\arrow["{e_1}"{description}, curve={height=-12pt}, from=3-4, to=1-4]
	\arrow["{e_2}"{description}, curve={height=-12pt}, from=1-4, to=3-4]
	\arrow["{e_7}"{description}, from=3-4, to=4-3]
	\arrow["{e_6}"{description}, from=3-2, to=4-3]
	\arrow["{r_7}"{description}, from=1-8, to=1-6]
	\arrow["{r_2}"{description}, curve={height=-12pt}, from=1-6, to=3-6]
	\arrow["{r_6}"{description}, from=3-6, to=4-7]
	\arrow["{r_3}"{description}, from=4-7, to=3-8]
	\arrow["{r_4}"{description}, curve={height=-12pt}, from=3-8, to=1-8]
	\arrow["{r_5}"{description}, curve={height=-12pt}, from=1-8, to=3-8]
	\arrow["{r_1}"{description}, curve={height=-12pt}, from=3-6, to=1-6]
\end{tikzcd}\]
We computed that $\sgn_{\vphi}(e_i) = -1$ when $i$ is $3, 4$, or $6$, and that $\sgn_{\vphi}(e_i) = 1$ otherwise. Let $U$ be the following orientation on $\hat G$.

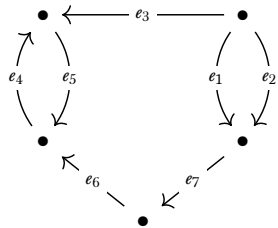
\[\begin{tikzcd}
	\bullet && \bullet \\
	\\
	\bullet && \bullet \\
	& \bullet
	\arrow["{e_3}"{description}, from=1-3, to=1-1]
	\arrow["{e_4}"{description}, curve={height=-12pt}, from=3-1, to=1-1]
	\arrow["{e_5}"{description}, curve={height=-12pt}, from=1-1, to=3-1]
	\arrow["{e_1}"{description}, curve={height=12pt}, from=1-3, to=3-3]
	\arrow["{e_2}"{description}, curve={height=-12pt}, from=1-3, to=3-3]
	\arrow["{e_7}"{description}, from=3-3, to=4-2]
	\arrow["{e_6}"{description}, from=4-2, to=3-1]
\end{tikzcd}\]
We have $\sgn_{U}(e_i) = -1$ when $i$ is $1, 4$ or $6$, and $\sgn_{U}(e_i) = 1$ otherwise. It follows that $\sgn_{\vphi(U)}(r_i) = -1$ when $i$ is $1, 3$ or $6$, and $\sgn_{\vphi(U)}(r_i) = 1$ otherwise.
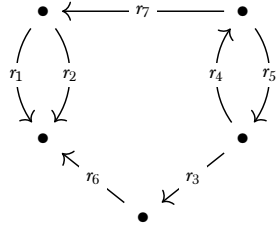
\[\begin{tikzcd}
	\bullet && \bullet \\
	\\
	\bullet && \bullet \\
	& \bullet
	\arrow["{r_7}"{description}, from=1-3, to=1-1]
	\arrow["{r_2}"{description}, curve={height=-12pt}, from=1-1, to=3-1]
	\arrow["{r_6}"{description}, from=4-2, to=3-1]
	\arrow["{r_3}"{description}, from=3-3, to=4-2]
	\arrow["{r_4}"{description}, curve={height=-12pt}, from=3-3, to=1-3]
	\arrow["{r_5}"{description}, curve={height=-12pt}, from=1-3, to=3-3]
	\arrow["{r_1}"{description}, curve={height=12pt}, from=1-1, to=3-1]
\end{tikzcd}\]
To see what this is doing, take a simple cycle $C$ in $G$. Proceed along it in the direction agreeing with an edge $\ell$. If we follow $\vphi(C)$ in the direction agreeing with $\vphi(\ell)$, then $\vphi(C)$ agrees with an edge $\vphi(e_i)$ precisely when $C$ agrees with $e_i$. The next result is a useful form of this observation.
\end{example}

\begin{lemma} 
	Let $\vphi: (G, \Gamma, e) \to (H, \Gamma', w)$ be a morphism in $\orCyc$ and $U$ be a partial orientation on $G$. Then the image of any oriented cycle in $U$ under $\vphi$ is an oriented cycle in $\vphi_{\mathbb O}(U)$. 
	\end{lemma}
	
	\begin{proof}
	Suppose $C$ is a simple oriented cycle in $U$. Let $\alpha(C) = \sum_i \sgn_U(e_i) h_{e_i}$. Then \begin{align*}
		\vphi_*(\alpha(C)) &= \sum_i \sgn_\vphi(e_i)\sgn_U(e_i) h_{\vphi(e_i)} \\
		&= \sum_i \sgn_{\vphi_{\mathbb O}(U)}(\vphi(e_i)) h_{\vphi(e_i)}
	\end{align*} 
	By definition $\vphi_*$ takes algebraic cycles to algebraic cycles, so proceeding along $\vphi(C)$ in some direction must yield a cycle $C'$ with $\alpha(C') = \sum_i \sgn_{\vphi_{\mathbb O}(U)}(\vphi(e_i)) h_{\vphi(e_i)}$, as desired.
	\end{proof}
	
	\begin{cor} \label{preservation of oriented cycles}
	Any $\vphi_{\mathbb O}$ yields bijections $\mc N_{\mathbb O}(G) \cong \mc N_{\mathbb O}(H)$ and $\mc S_{\mathbb O}(G) \cong \mc S_{\mathbb O}(H)$.
	\end{cor}

Given the connection between orientations and divisors, there ought to be some connection between $\vphi_{\mathbb O}$ and $\vphi_*$. In order to describe it, we first extend the definition of the Picard group to the continuous case.

\begin{defn}
Let $(\hat G, \hat e)$ be a based directed graph. We let $\Pic(G) := J(G) \times \Z$. We have a sequence consisting of an isomorphism followed by two inclusions. \[ \Pic^0_{\dee}(G) \to J_{\dee}(G) \to J(G) \to \Pic(G) \] 
We wish to extend the composition $\Pic^0_{\dee}(G) \to \Pic(G)$, which we will temporarily denote by $\iota'$, to a map $\iota_{\hat e}: \Pic_{\dee}(G) \to \Pic(G)$. We set \[\iota_{\hat e}(D) = (\iota'(D - n t(\hat e)), n) \] 
As the subscript suggests, this depends on $\hat e$. \end{defn}

\begin{remark} \label{discrete is only obstacle}
    We now have enough terminology to outline our general strategy. Since the Chern class of an orientation sits in $\Pic^{g-1}_{\dee} (G)$, we have linked $\vphi_{\mathbb O}$ and $(\vphi_* \times \id_Z): J(G)_{\dee} \times \{g-1\} \to J(H)_{\dee} \times \{g-1\}$. Suppose we have a morphism $\vphi: (G, \Gamma, e) \to (H, \Gamma', w)$ in $\orCyc$. Then we can consider the following diagram.

\[\begin{tikzcd}
	{\overline{\mathbb O}(G)} && {\overline{\mathbb O}(H)} \\
	{\Pic^{g-1}_{\dee} (G)} && {\Pic^{g-1}_{\dee} (H)} \\
	{J_{\dee}(G) \times \{g-1\}} && {J_{\dee}(H) \times \{g-1\}}
	\arrow["{\vphi_{\mathbb O}}", from=1-1, to=1-3]
	\arrow["c", from=1-3, to=2-3]
	\arrow["{\iota_{\hat w}}", from=2-3, to=3-3]
	\arrow["c"', from=1-1, to=2-1]
	\arrow["{\iota_{\hat e}}"', from=2-1, to=3-1]
	\arrow["{\vphi_*}"', from=3-1, to=3-3]
\end{tikzcd}\]

\Cref{geometric effectiveness}, along with the definitions, shows that following sourceless orientations from $\overline{\mathbb O}(G)$ down the left side (but not across) yields exactly the elements of $\theta_{\hat e} \times \{g-1\} \subset J_{\dee}(G) \times \{g-1\}$. On the other hand combining \cref{geometric effectiveness} with \cref{preservation of oriented cycles} shows that following the sourceless orientations from $\overline{\mathbb O}(G)$ along the top and then down the right side of the diagram yields exactly the elements of $\theta_{\hat w} \times \{g-1\} \subset J_{\dee}(H) \times \{g-1\}$. Motivated by this, we will characterize preservation of the theta divisor by $\vphi_*$ in terms of commutativity or failure of commutativity of this diagram.
\end{remark}

\begin{defn} \label{rigidity definition}
Let $\vphi: (G, \Gamma, e) \to (H, \Gamma', w)$ be a morphism in $\orCyc$. Consider the following diagram.

\[\begin{tikzcd}
	{\mathbb O(G)} && {\mathbb O(H)} \\
	{\Pic^{g-1}_{\dee} (G)} && {\Pic^{g-1}_{\dee} (H)} \\
	{J_{\dee}(G) \times \{g-1\}} && {J_{\dee}(H) \times \{g-1\}}
	\arrow["{\vphi_{\mathbb O}}", from=1-1, to=1-3]
	\arrow["c", from=1-3, to=2-3]
	\arrow["{\iota_{\hat w}}", from=2-3, to=3-3]
	\arrow["c"', from=1-1, to=2-1]
	\arrow["{\iota_{\hat e}}"', from=2-1, to=3-1]
	\arrow["{\vphi_*}"', from=3-1, to=3-3]
\end{tikzcd}\]
We call $\vphi$ \textit{rigid} if this diagram commutes.
\end{defn}

Because we have not yet proven that $\vphi_{\mb O}: \overline{\mathbb O}(G) \to \overline{\mathbb O}(H)$ is well defined, the above definition uses $\mathbb O(G)$ and $\mathbb O(H)$ instead of $\overline{\mathbb O}(G)$ and $\overline{\mathbb O}(H)$. Once this is proven the choice will be of no significance. 

An alternative characterization would be that rigid morphisms are the morphisms of the largest subcategory of $\orCyc$ for which $\iota \circ c$ is a natural transformation $\mb O(-) \Rightarrow J_{\dee}(-) \times \{g-1\}$.

\begin{defn}
Given a morphism $\vphi: (G, \Gamma, e) \to (H, \Gamma', w)$ in $\orCyc$, we define the \textit{rigidity divisor} $\mc E_\vphi \in J_{\dee}(H)$ by \[\mc E_\vphi := \vphi_*\iota_{\hat e}c(\Gamma) - \iota_{\hat w} c(\Gamma') + \sum_{\sgn_\vphi(\ell) = -1} h_{\vphi(\ell)} \] 
Given $X = \{e_1, e_2, \dots, e_r \} \subseteq E(G)$, we further define the \textit{lowering divisor} $\mathbb L_X \in J_{\dee}(G)$ by \[\mathbb L_X := \sum_i (P_{t (\vphi (e_i))} - \vphi_*(P_{t(e_i)}))\]
\end{defn}

We will show in \cref{diagram rigidity} that taking any orientation on $G$, following it along the two possible paths to $\Pic_{\dee}^{g-1}(H)$ in \cref{rigidity definition}, and taking the difference of the two results yields $\mc E_\vphi$. Deorienting edges (or equivalently biorienting them, cf \cref{deorient biorient}) is useful to us and appears in the literature. Doing so introduces further error beyond that indicated by the rigidity divisor, the amount of which is quantified by the lowering divisor.

We are now in a position to state our current goal, which we will prove in \cref{proof of rigidity equivalences}. 

\begin{theorem} \label{extended torelli}
Let $\vphi: (\hat G, \hat e) \to (\hat H, \hat w)$ be a morphism in $\orCyc$ between graphs of genus at least two. Then the following are equivalent. \begin{enumerate}
    \item $\vphi$ is rigid.
    \item $\mc E_\vphi \sim 0$.
    \item $\vphi_*(\theta_{\hat e}) = \theta_{\hat w}$.
    \item $\vphi_*(\im S_{\hat e}^1) = \im S_{\hat w}^1$.
\end{enumerate}
\end{theorem}

\begin{remark} \label{special bijections}
This is an opportune moment to make explicit how we can compute with the maps of the above diagram.

Covering $G$ with cycles yields the coefficients $\sgn_\vphi(\ell)$. These directly define $\vphi_{\mathbb O}$ and are sufficient to define $\vphi_*$ via $h_\ell \mapsto \sgn_\vphi(\ell)h_{\vphi(\ell)}$. 

The Chern class map is easy to compute. Its inverse $\OO$ is less friendly. The $\Pic^0_{\dee}(G)$ torsor structure on $\overline{\mathbb{O}}$ induced via the Chern class map has an accessible intepretation: given $[U]$ and vertices $p, q$, the set of vertices accessible from $p$ by oriented paths in $U$ either includes $q$ or has an oriented cut pointing toward it; in the second case we reverse the cut. Repeating this process yields an oriented path from $p$ to $q$. Then $p - q + [U]$ is the class of taking this orientation and reversing all the edges along the path. We can use this and $\OO(U) = c(U) - c(\Gamma) + [\Gamma]$ to describe $\OO(U)$. For an algorithmic description which extends to partial orientations see \cite{backman_riemann-roch_2017}.

The map $\iota_{\hat e}$ can be informally described as follows. For any $D$, arbitrarily pair each $-1$ chip $q_i$ of $D - \deg(D)t(\hat e)$ with a $+1$ chip $p_i$. Choose any path $P_i$ from $q_i$ to $p_i$. Then $\iota_{\hat e}(D) = (\sum_i \alpha(P_i), \deg D)$. On the other hand, $\iota_{\hat e}\inv (\sum_{\ell} a_\ell h_\ell, k) = kt(\hat e) + \sum_\ell a_\ell (t(\ell) - o(\ell))$. In practice, computations in $\Pic(G)$ are often most easily done by applying $\iota_{\hat e}\inv$ and using discrete methods found in sources such as \cite{klivans_mathematics_2018} or \cite{baker_chip-firing_2013}.

When $D$ is effective, we have $\iota_{\hat e}(D) = S_{\hat e}^k(D)$, where on the right we regard $D$ as an element of $V(G)^{(k)}$. In particular, our setup is designed so that following the bijections $\mathbb O(G) \cong \Pic^{g-1}_{\dee} \cong J_{\dee}(G) \times \{g-1\}$ yields bijections $\mc S_{\mathbb O}(G) \cong \mc S(G) \cong \theta_{\hat e} \times \{g-1\}$, and likewise with their complements.
\end{remark}

The equivalence of the first two items in \cref{extended torelli} will follow from the next, more general result.

\begin{theorem} \label{diagram rigidity}

Let $\vphi: (G, \Gamma, e) \to (H, \Gamma', w)$ be a morphism in $\orCyc$. Let $X = \{e_1, e_2, \dots, e_r \} \subseteq E(G)$. Consider the following diagram.
\[\begin{tikzcd}
	{\mathbb O(G, X)} && {\mathbb O(H, \vphi(X))} \\
	{\Pic^{g - 1-r}_{\dee} (G)} && {\Pic^{g - 1-r}_{\dee} (H)} \\
	{J_{\dee}(G) \times \{g-1-r\}} && {J_{\dee}(H) \times \{g-1-r\}}
	\arrow["{\vphi_{\mathbb O}}", from=1-1, to=1-3]
	\arrow["c", from=1-3, to=2-3]
	\arrow["{\iota_{\hat w}}", from=2-3, to=3-3]
	\arrow["c"', from=1-1, to=2-1]
	\arrow["{\iota_{\hat e}}"', from=2-1, to=3-1]
	\arrow["{\vphi_*}"', from=3-1, to=3-3]
\end{tikzcd}\]
For any $U \in \mathbb O(G, X)$, we have $(\vphi_* )\iota_{\hat e}c(U) - \iota_{\hat w} c \vphi_{\mathbb O}(U) = \mc E_\vphi + \mathbb L_X$. 
\end{theorem}

\begin{proof}
We have
\begin{align*}
    U &\xmapsto{c} \sum_{\sgn_U(\ell) = 1} t(\ell) + \sum_{\sgn_U(\ell) = -1} o(\ell) - \sum_{v \in V(G)} v \\
    &= \sum_{\sgn_U(\ell) = 1} t(\ell) + \sum_{\substack{\sgn_U(\ell) = -1 \\ \sgn_\vphi(\ell) = 1}} o(\ell) + \sum_{\substack{\sgn_U(\ell) = -1 \\ \sgn_\vphi(\ell) = -1}} o(\ell) - \sum_{v \in V(G)} v \\
    &= \left (\sum_{\ell \in E(G)} t(\ell) - \sum_{v \in V(G)} v \right)-  \sum_{\sgn_U(\ell) = 0} t(\ell)  + \sum_{\substack{\sgn_U(\ell) = -1 \\ \sgn_\vphi(\ell) = 1}}\left(o(\ell) - t(\ell) \right) + \sum_{\substack{\sgn_U(\ell) = -1 \\ \sgn_\vphi(\ell) = -1}}\left(o(\ell) - t(\ell) \right) \\
    &= c(\Gamma) -  \sum_{\sgn_U(\ell) = 0} t(\ell) + \sum_{\substack{\sgn_U(\ell) = -1 \\ \sgn_\vphi(\ell) = 1}}\left(o(\ell) - t(\ell) \right) + \sum_{\substack{\sgn_U(\ell) = -1 \\ \sgn_\vphi(\ell) = -1}}\left(o(\ell) - t(\ell) \right) \\
    &\xmapsto{\iota_{\hat e}} \iota_{\hat e} c(\Gamma) -  \sum_{\sgn_U(\ell) = 0} P_{t(\ell)} + \sum_{\substack{\sgn_U(\ell) = -1 \\ \sgn_\vphi(\ell) = 1}}-h_\ell + \sum_{\substack{\sgn_U(\ell) = -1 \\ \sgn_\vphi(\ell) = -1}} -h_\ell \\
    &\xmapsto{\vphi_* } (\vphi_* ) \iota_{\hat e} c(\Gamma) - \sum_{\sgn_U(\ell) = 0} \vphi_* P_{t(\ell)} + \sum_{\substack{\sgn_U(\ell) = -1 \\ \sgn_\vphi(\ell) = 1}}-h_{\vphi(\ell)} + \sum_{\substack{\sgn_U(\ell) = -1 \\ \sgn_\vphi(\ell) = -1}} h_{\vphi(\ell)}
\end{align*}
and
\begin{align*}
    U &\xmapsto{c \vphi_{\mathbb O}} \sum_{\sgn_U(\ell)\sgn_\vphi(\ell) = 1} t(\ell) + \sum_{\sgn_U(\ell)\sgn_\vphi(\ell) = -1} o(\ell) - \sum_{v \in V(H)} v \\ 
    &= \sum_{\ell \in E(G)} t(\vphi(\ell)) - \sum_{v \in V(H)} v -  \sum_{\sgn_U(\ell) = 0} t(\vphi (\ell)) + \sum_{\sgn_U(\ell)\sgn_\vphi(\ell) = -1} (o(\vphi(\ell)) - t(\vphi(\ell)))  \\
    &= c(\Gamma') -  \sum_{\sgn_U(\ell) = 0} t(\vphi (\ell)) + \sum_{\sgn_U(\ell)\sgn_\vphi(\ell) = -1} (o(\vphi(\ell)) - t(\vphi(\ell)))  \\
    &\xmapsto{\iota_{\hat w}} \iota_{\hat w} c(\Gamma')  - \sum_{\sgn_U(\ell)\sgn_\vphi(\ell) = -1} h_{\vphi(\ell)} -  \sum_{\sgn_U(\ell) = 0} P_{t(\vphi (\ell))}
\end{align*}
so
\begin{multline*}
    \vphi_* \iota_{\hat e}c(U) - \iota_{\hat w} c \vphi_{\mathbb O}(U) = (\vphi_* ) \iota_{\hat e} c(\Gamma) + \sum_{\substack{\sgn_U(\ell) = -1 \\ \sgn_\vphi(\ell) = 1}}-h_{\vphi(\ell)} + \sum_{\substack{\sgn_U(\ell) = -1 \\ \sgn_\vphi(\ell) = -1}} h_{\vphi(\ell)} - \\ \left(\iota_{\hat w} c(\Gamma') + \sum_{\sgn_U(\ell)\sgn_\vphi(\ell) = -1} -h_{\vphi(\ell)} \right) + \mathbb L_X
\end{multline*}
This is precisely
\begin{align*}
    \vphi_* \iota_{\hat e}c(U) - \iota_{\hat w} c \vphi_{\mathbb O}(U) &= (\vphi_* )\iota_{\hat e} c(\Gamma) - \iota_{\hat w} c(\Gamma') + \sum_{\sgn_\vphi(\ell) = -1} h_{\vphi(\ell)} + \mathbb L_X \\
    &= \mathcal E_{\vphi} + \mathbb L_X
\end{align*}
\end{proof}

\begin{cor} \label{lowering compatibility}
If $X$ is any (possibly empty) set of the edges $\hat e_i$ pointing to $t(\hat e)$, all of which map to edges pointing to $t(\hat w)$, then $\mathbb L_X = 0$. Consequently, for such an $X$ the following are equivalent: \begin{enumerate}
    \item The diagram commutes for every such $X$.
    \item The diagram commutes for some such $X$.
    \item $\mc E_\vphi \sim 0$.
\end{enumerate}
\end{cor}

\begin{cor} \label{well defined orientation pushforward}
The map $\vphi_{\mathbb O}$ is well defined on $\overline{\mathbb O}(G, X)$ for any $X$.
\end{cor}
\begin{proof}
We need to show that if $[U] = [U'] \in \overline{\mathbb O}(G, X)$, or equivalently $c(U) = c(U')$, then $[\vphi_{\mathbb O}(U)] = [\vphi_{\mathbb O}(U')]$. We compute that \begin{align*}
    \iota_{\hat w} c(\vphi_{\mathbb O}(U)) - \iota_{\hat w} c(\vphi_{\mathbb O}(U')) &= ((\vphi_* \times \id_Z)\iota_{\hat e} c(U) - \mc E_\vphi - \mathbb L_X) - ((\vphi_* \times \id_Z)\iota_{\hat e} c(U') - \mc E_\vphi - \mathbb L_X) \\
    &= (\vphi_* \times \id_Z)\iota_{\hat e} c(U) - (\vphi_* \times \id_Z)\iota_{\hat e} c(U') \\
    &= 0
\end{align*}
as desired.
\end{proof}

\begin{cor}
Let $\vphi: (G, \Gamma, e) \to (H, \Gamma', w)$ and $\psi: (H, \Gamma, w) \to (I, \Gamma', \ell)$ be morphisms in $\orCyc$. Then $\mc E_{\psi\vphi} \sim \mc E_\psi + (\psi_* ) \mc E_\vphi$.
\end{cor}
\begin{proof}
A diagram chase with $X = \emptyset$.
\end{proof}

\begin{remark} \label{deorient biorient}
One sometimes sees in the literature orientations in which edges are allowed to be bioriented (for example, the ``$1$-orientations'' of \cite{caporaso_combinatorics_2019}, in which one biorients a base edge), and the theory of break divisors \cite{baker-yao_bernardi-process_2017} can be interpreted in terms of orientations with a single bioriented edge. These setups can be more closely related to the above, and to the existing theory of partial orientations, in a way we now sketch. See Backman and Hopkin's work in \cite{backman2017fourientations} for related discussions.

When $m = g-1 + m'$ for $m' \geq 0$, let $\mathbb O^{m}(G)$ denote the set of orientations on $G$ for which exactly $m'$ edges are bioriented, and all the rest are oriented in the ordinary sense. The Chern class map can be taken in this case in the evident way, and we can quotient to obtain $\overline{\mathbb O}^m(G)$. We can regard $\mathbb O^{m}(-)$ as a functor on $\orCyc$, but to make $\overline{\mathbb O}^m(G)$ functorial would require a category with more data. 

We can interchange elements of $\mathbb O^{m}(G)$ and $\mathbb O^{g - 1 - m'}(G)$ by reversing each oriented edge, making unoriented edges be bioriented, and making bioriented edges be unoriented. Denote this interchange by $*$. This map is a natural isomorphism $\mathbb O^{m}(-) \cong \mathbb O^{g-1-m'}(-)$. One checks that we have $c(* U) = K - c(U)$, where $K$ is the canonical divisor. This makes it straightforward to dualize results concerning partial orientations to results concerning bioriented edges.
\end{remark}

\subsection{Preservation of the First Abel-Jacobi Image} \label{Preservation of the First Abel-Jacobi Image}

In this subsection we prove the following technical lemma, using standard methods of argument.

\begin{lemma} \label{theta divisor can't be translated onto itself}
Let $G$ be a $2$-connected and $2$-edge connected graph of genus at least two. Let $D \in \Pic^0_{\dee}(G)$ be nonzero. Then $D + \mc S \neq \mc S$.
\end{lemma}

Note that $\mathcal S$ is a translation of any choice of discrete theta divisor. Translating $\mathcal S$ back into $\Pic^0_{\dee}(G)$ and setting $D$ to be a rigidity divisor will be a key step in characterizing rigidity in terms of preservation of the discrete theta divisor.

\begin{defn}
Let $(f)$ be a principal divisor on $G$. We define $M(f)$ to be the set of vertices of $G$ on which $f$ is maximal and $N(f)$ to be the set of edges on which $f$ is minimal. 
\end{defn}
\begin{defn}
Given a subset $X \subseteq V(G)$, let $\partial X$ be the smallest subgraph of $G$ containing all edges between vertices in $X$ and vertices not in $X$. For $x \in V(G)$, let $\partial_x X$ be the set of edges between $x$ and vertices not in $X$.
\end{defn}

For the next result and its first corollary, a dual result holds for $D'$ and $N(f)$, obtained by using the observation that $N(f) = M(-f)$.

\begin{prop}
Let $D', D \in \Div(G)$ be effective, and let $D' = (f) + D$ for $(f) \neq 0$. For all $x \in M(f)$, we have $D(x) \geq |\partial_x M(f)|$. 
\end{prop}
\begin{proof}
Observe that \begin{align*}
    -D(x) \leq (f)(x) &= \sum_{\text{edges between }x\text{ and }x'} f(x') - \deg (x) \cdot f(x) \\
    &= \sum_{e \in \partial_x M(f)\text{ connected to } x' \neq x} f(x') - |\partial_xM(f)| \cdot f(x) \\
    & \leq -|\partial_xM(f)|
\end{align*}
as desired.
\end{proof}
\begin{cor} \label{high connectedness controls degree and boundary size}
Let $G$ be $k$-edge connected. Then in the above situation $\deg D \geq |E(\partial M(f))| \geq k$. Further, if $\deg D = k$, then $G$ is $k$-edge connective, $\deg D = |E(\partial M(f))|$, and for every vertex $x$ we have $D(x) = |\partial_x M(f)|$. In particular, the support of $D$ is exactly $V(\partial M(f))$.
\end{cor}
\begin{proof}
One obtains $\deg D \geq |E(\partial M(f))|$ by using the proposition and summing across $x \in M$. If $|E(\partial M(f))| \geq k$ did not hold, we could remove all the edges in $\partial M(f)$ from $G$ and disconnect $G$. The rest follows immediately.
\end{proof}

\begin{remark}
One can interpret the above by thinking of values of $f$ as specifying pressure at each vertex, capable of pushing the `chips' of chip firing. Qualitatively, the above conditions on $D' = (f) + D$ say that the only way we can push an effective divisor $D$ to another effective divisor is by finding regions of the graph that $D$ seals off and pressurizing a selection of them. Quantitatively, they specify that $D$ requires one chip in order to seal off each outgoing edge from the region.
\end{remark}

\begin{lemma} \label{first abel jacobi image can't be translated to itself}
Let $G$ be a $2$-connected and $2$-edge connected graph and let $D \in \Pic^0_{\dee}(G)$ with $D \not \sim 0$. Suppose $v_i \mapsto v_{\sigma(i)}$ is a permutation of the vertices of $G$ such that $D + v_i \sim v_{\sigma(i)}$ for each $i$. Then $G$ is a single cycle.
\end{lemma}

\begin{proof}
Note that translating by $D$ does not fix any $v_i$. Pick any $v_i$. We have $D \sim v_i - v_{\sigma\inv(i)} \sim v_{\sigma(i)} - v_i$, so $2v_i \sim v_{\sigma(i)} + v_{\sigma\inv(i)}$. By \cref{high connectedness controls degree and boundary size}, we must have that $G$ is $2$-edge connective. Pick an $f$ with $2v_i = (f) + v_{\sigma(i)} + v_{\sigma\inv(i)}$. Again by \cref{high connectedness controls degree and boundary size}, we have $|E(\partial N(f))| = 2$, and for $x \neq v_i$ we have $|\partial_x N(f)| = 0$. Thus there are exactly two edges between $N(f)$ and the rest of $G$, each of which is adjacent to $v_i$. If $N(f)$ contained any vertex $t \neq v_i$ it would follow that removing $v_i$ disconnects $t$ from the rest of $G$, violating $2$-connectedness of $G$. Thus $N(f) = \{ v_i \}$, from which we see that $\deg(v_i) = 2$. Since $v_i$ was arbitrary every vertex in $G$ has degree at most two, which since $G$ is $2$-connected implies that $G$ is a single cycle.
\end{proof}

The following result generalizes the claim in \cref{geometric effectiveness}.

\begin{theorem} (Backman) \label{ineffective is dominated by acyclic}
Let $Q$ be a divisor of degree at most $g-1$. Then $|Q| = \emptyset$ if and only if there exists an acyclic partial orientation $U$ and some $A \sim Q$ for which $c(U) \geq A$. On the other hand $Q$ is linearly equivalent to an effective $B$ if and only if there exists a sourceless partial orientation $W$ and an effective $B' \sim Q$ with $c(W) = B'$.
\end{theorem}
\begin{proof}
A proof and an algorithm for producing the relevant objects appear in \cite{backman_riemann-roch_2017}.
\end{proof}

\begin{cor} \label{extend to nonspecial}
Let $Q$ be a divisor of degree at most $g-1$ with $|Q|$ empty. Then there exists an effective divisor $T$ such that $Q + T \in \mc N(G)$.
\end{cor}

\begin{proof}
Apply \cref{ineffective is dominated by acyclic} to find an acyclic partial orientation $U$ with $c(U) \geq Q$. This partial orientation defines a partial order on the vertices of $G$ by reachability along oriented paths. Extend this partial order to a total order $\preceq$. We can orient the edges of $G$ to always point to the $\preceq$-greater vertex, producing an extension of $U$ to an acyclic full orientation $U'$. Since $c(U') \geq c(U) \geq Q$, we have $c(U') - Q \geq 0$, the desired $T$ is $c(U') - Q$.
\end{proof}

We can now prove \cref{theta divisor can't be translated onto itself}.

\begin{proof}
By hypothesis $G$ is not cyclic, so by \cref{first abel jacobi image can't be translated to itself} there is some vertex $v$ with $D + v$ not linearly equivalent to an effective divisor. Applying \cref{extend to nonspecial} to find a $T \geq 0$ with $D + v + T \in \mc N(G)$, we have $v + T \in \mc S(G)$ and $D + v + T \not \in \mc S(G)$.
\end{proof}

We need one last technical result.

\begin{prop} \label{same size AJ images}
	Let $(G, \hat e)$ and $(H, \hat w)$ be graphs with $M(G) \cong M(H)$. Then for any $d \geq g(G) - |V(G)|$, the number of linear equivalence classes of degree $d$ containing an effective divisor is the same in both $G$ and $H$. Equivalently, $|\im S^d_{\hat e}| = |\im S^d_{\hat w}|$. 
	\end{prop}
	\begin{proof}
	By \cref{partially orientable}, every divisor of these degrees is partially orientable. Applying \cref{ineffective is dominated by acyclic} and \cref{preservation of oriented cycles} gives the result.
	\end{proof}
	
	\begin{remark}
	It would be extremely strange if the condition on the degrees in this corollary were necessary, but the author does not know how to remove it.
	\end{remark}

\subsection{Proof of equivalent characterizations of rigidity} \label{proof of rigidity equivalences}

In \cref{discrete is only obstacle} we outlined how failure to preserve the discrete theta divisor ought to result in a nonzero rigidity divisor. The prior section gave us the technical tools necessary to demonstrate that the presence of a nonzero rigidity divisor always prevents preservation of the discrete theta divisor. We now prove \cref{extended torelli}.

\begin{proof} $(1) \Leftrightarrow (2)$: We have already proven, in \cref{extended torelli}, that $\vphi$ is rigid if and only if $\mc E_\vphi \sim 0$.

$(1) \Leftrightarrow (3)$: We have 
\begin{align*} \vphi_*(\theta_{\hat e}, g-1) &= \vphi_*(\iota_{\hat e}c(\mc O \iota_{\hat e}\inv (\theta_{\hat e}, g-1))) \\
	&= \vphi_* \iota_{\hat e}c (\mc O \iota_{\hat e}\inv (\theta_{\hat e}, g-1)) \\
	&= \iota_w c \vphi_{\mathbb O} (\OO \iota_e\inv (\theta_{\hat e}, g-1)) + \mc E_\vphi \\
	&= \iota_w c \vphi_{\mathbb O}(\mc S_{\mathbb O}(G)) + \mc E_\vphi \\
	&= \iota_w c (\mc S_{\mathbb O}(H)) + \mc E_\vphi \\
	&= (\theta_{\hat w}, g-1) + \mc E_\vphi \\
	&= (\theta_{\hat w} + \mc E_\vphi, g-1) \end{align*} 
If $\vphi$ is rigid, then from $(1) \Leftrightarrow (2)$ we know that $\mc E_\vphi \sim 0$, so $\vphi_*(\theta_{\hat e}) = \theta_{\hat w}$. If $\vphi$ is not rigid, then again by $(1) \Leftrightarrow (2)$ we know that $\mc E_\vphi \not \sim 0$. Now $\iota_w\inv((\theta_{\hat w}, g-1) + \mc E_\vphi) = \mc S(H) + \iota_w\inv(\mc E_\vphi)$, where $\iota_w\inv(\mc E_\vphi) \not \sim 0$. By \cref{theta divisor can't be translated onto itself}, we have $\mc S(H) + \iota_w\inv(\mc E_\vphi) \neq \mc S(H)$, so applying $\iota_w$, we find that $\vphi_*(\theta_{\hat e}) \neq \theta_{\hat w}$.

$(4) \Rightarrow (3)$: Immediate.

$(3) \Rightarrow (4)$: If $g = 2$ then the discrete theta divisors are the same as $\im S_{\hat e}^1$ and $\im S_{\hat w}^1$ and so there is nothing to show, so assume $g > 2$. 

We start by showing that we must have $\vphi_*(\im S_{\hat e}^{g-2}) = \im S_{\hat w}^{g-2}$. Both sets are finite, and by \Cref{same size AJ images} are of the same size, so it is sufficient to show that $\vphi_*(\im S_{\hat e}^{g-2}) \subseteq \im S_{\hat w}^{g-2}$. 

Since $\vphi$ is rigid, the diagram of \cref{diagram rigidity} commutes with $X = \{ e\}$. If $W \in \mathbb O(G, X)$ is acyclic, then by \cref{preservation of oriented cycles} so is $\vphi_{\mathbb O}(W)$. It follows that for any $U \in \mathbb O(G, X)$ with an effective divisor in its Chern class, we have that $\vphi_{\mathbb O}(U)$ is not acyclic, so by \Cref{ineffective is dominated by acyclic} we know that $\vphi_{\mathbb O}(U)$ is equivalent to a sourceless orientation. Thus if $D \sim c(U) \in \Pic^{g-2}_{\dee}(G)$ is effective, we have that $\iota_{\hat w}\inv (\vphi_* ) \iota_{\hat e}(D) \sim c(\vphi_{\mathbb O}(U))$, implying that $\iota_{\hat w}\inv (\vphi_* ) \iota_{\hat e}(D)$ is linearly equivalent to an effective divisor.

Now if $D \in \Pic^{g-2}_{\dee}(G)$ is effective, then by \cref{ineffective is dominated by acyclic} we can find some $U \in \mathbb O^{g-2}(G)$ with Chern class $D$. Necessarily, $U$ only has a single unoriented edge $r$. Letting $R$ be the set of vertices reachable from $t(\hat e)$ along oriented paths, if $R$ does not include one of the ends of $r$, the edges from $R$ to $R^c$ form a consistently oriented cut. Reverse the cut and repeat until there is an oriented path from $t(\hat e)$ to $r$. Using the third operation of \cref{linear equivalence of orientations} repeatedly along this path, we find that $U$ is equivalent to an orientation in $\mathbb O(G, X)$. Thus $\iota_{\hat w}\inv (\vphi_* ) \iota_{\hat e}(D)$ is linearly equivalent to an effective divisor, so applying $\iota_{\hat w}$ shows that $\im S_{\hat e}^{g-2}$ is preserved.

Suppose for the sake of contradiction that we can choose a vertex $p \in G$ with $\vphi_*(S^1_{\hat e}(p)) \not \in \im S^1_{\hat w}$. Let $Q = \iota_{\hat w}\inv(\vphi_* )\iota_{\hat e}(p)$. By assumption $Q$ is not linearly equivalent to any effective divisor. By \cref{extend to nonspecial} we can find some effective $T \in \Pic^{g-2}_{\dee}$ for which $Q + T \in \mc N(H)$. Writing $T = \iota_{\hat w}\inv (\vphi_* ) \iota_{\hat e}(D)$ for an effective $D$, we have $Q + T \sim \iota_{\hat w}\inv (\vphi_* ) \iota_{\hat e}(p + D) \in \mc S(H)$, which is the desired contradiction.
\end{proof}

\section{Interpretation in terms of the structure of graphs} \label{Geometric interpretation and the main result}

In this section we prove the main result, \cref{geometric rigidity}. This characterizes rigidity of a morphism as coming from being `near to' a graph isomorphism, in a sense we will make precise. 

The first order of business is to describe the inability of $J(G)$ to distinguish certain edges. In \cite{caporaso_torelli_2010} maximal collections of such edges are referred to as \textit{codimension one sets}. We use the term series, which follows \cite{oxley_connectivity_1981}.

\begin{defn}
Let $G$ be a $2$-edge connected graph and $e \in G$ be an edge. Let $\varepsilon (e)$ be the set of edges $\ell$ such that $G \bs \{e, \ell \}$ is not connected. We refer to $\varepsilon(e)$ as a \textit{series class} and say $e$ and $\ell$ are \textit{in series}. We call $e$ a \textit{series edge} if $|\varepsilon(e)| > 1$. 
\end{defn}

Edges in a $2$-edge connected graph are in series precisely when they are contained in the same set of simple cycles. A $2$-edge connected graph is $3$-edge connected exactly when the graph has no series edges. Here is the picture we ought to have when a series class $C$ has at least two edges. Consider pearls whose positions on a single strand necklace are fixed. The components of $G \bs C$ are the pearls, while the series edges are the lengths of chain running between the pearls.

\begin{lemma} \label{series ambiguity}
Let $\ell$ be an edge in a $2$-connective and $2$-edge connected graph $G$, with ends $v,w$. Let $p,q$ be any two vertices in $G$ with $v-w \sim p-q$. Then $p$ and $q$ are the ends of an edge in series with $\ell$. Conversely, if $r$ is an edge in series with $\ell$, with ends $p, q$, then $v - w \sim \pm(p-q)$.
\end{lemma}

\begin{proof}
For the first claim, assume $v - w \neq p-q$, since otherwise there is nothing to prove. It follows that $v \neq p$. If $v = q$, then arguing as in the proof of \cref{first abel jacobi image can't be translated to itself} yields the conclusion. Reasoning similarly with $w$, we reduce to the case that all of $v,w,p$ and $q$ are distinct. 

We have $v + q = (f) + w + p$ for some nonconstant $f$. By \cref{high connectedness controls degree and boundary size}, we know that $V(\partial M(f)) = \{w,p \}$ and $V(\partial N(f)) = \{v,q\}$, and that each of these is adjacent to exactly one edge out of $M(f)$ or $N(f)$. By subtracting a constant we may assume the minimal value of $f$ is zero; this makes it clear that we cannot have that $v$ is adjacent to a vertex $t$ with $f(t) > 1$. Thus $f(w) = 1$, so $M(f) \cup N(f)$ partitions the vertices of the graph. The second edge $r$ out of $M(f)$ is therefore also the second edge into $N(f)$. Removing $\ell$ and $r$ then disconnects $G$, so $\ell$ and $r$ are in series as desired.

Conversely, if $\ell$ and $r$ are in series, say with $G \bs \{\ell, r\}$ having components $X$ and $Y$, one computes directly that setting $f$ to be $1$ on $X$ and $0$ on $Y$ yields the desired linear equivalence. 
\end{proof}

\begin{defn}
If $\psi: M(G) \to M(G)$ is a cyclic bijection which takes each series class into itself, we say that $\psi$ is \textit{series fixing}.
\end{defn}

\begin{prop} \label{series functorial ambiguity}
Let $\psi$ be a automorphism of $(\hat G, \hat e)$ in $\orCyc$. Then $\psi$ is series fixing if and only if $\psi_* = \id$.
\end{prop}

\begin{proof}
Suppose $\psi$ is series fixing. Let $\psi(\ell) = r$. Let $C$ be a simple cycle in $G$. If $C$ does not contain $\ell$, then $\ell^* \cdot \pi\alpha(C) = r^* \cdot \alpha(C) = 0$. Otherwise $\ell^* \cdot \alpha(C) = -r^* \cdot \alpha(C)$ if and only if $\sgn_\psi(\ell) = -1$. Since the $\alpha(C)$ span $H^1(\hat G, \R)$ it follows that $\psi_* h_\ell = \psi_* \pi(\ell^*) = \pi(\psi_*(\ell^*)) = \pi (\sgn_\psi(\ell) r^*) = \sgn_\psi(\ell) h_r = h_\ell$.

If $\psi_* = \id$, then for any edge $\ell$, if $\psi(\ell)$ were not in $\varepsilon(\ell)$, we could choose a cycle $C$ containing $\ell$ but not $\psi(\ell)$. Then $0 \neq h_\ell \cdot \alpha(C) = \psi_*(h_\ell) \cdot \alpha(C) = \sgn_\psi(\ell)h_{\psi(\ell)} \cdot \alpha(C) = 0$, a contradiction.
\end{proof}

\begin{cor}
We have $h_\ell = \pm h_r$ for edges $\ell$ and $r$ if and only if $\ell$ and $r$ are in series.
\end{cor}

The above results reflect the general tendency for the Jacobian to have a difficult time telling edges in series apart. Thus $3$-edge connectedness hypotheses result from the need for some kind of useful canonicity. When our graphs are $3$-edge connected, rigidity has the strongest possible meaning.

\begin{theorem} \label{geometric rigidity}
Let $\vphi: (G, \Gamma_1, \hat e) \to (H, \Gamma_2, \hat w)$ be a morphism in $\orCyc$ between graphs of genus at least two. Then $\vphi$ is rigid if and only if there exists a series fixing automorphism $\psi: (M(H), w) \to (M(H), w)$ such that the composition
\[ (M(G), e) \xrightarrow{\vphi} (M(H), w) \xrightarrow{\psi} (M(H), w) \] 
is an edge isomorphism. In particular, if $\vphi$ is rigid, then $G$ and $H$ are isomorphic graphs.
\end{theorem}
\begin{proof}
To check the backward direction, consider that $\vphi_* = \psi_*\vphi_* = (\psi \vphi)_*$ by \cref{series functorial ambiguity}. An edge isomorphism takes $\im S^1_{\hat e}$ to $\im S^1_{\hat w}$, as we can check by hand on paths, so by $(4) \Rightarrow (1)$ in \cref{extended torelli} we have that $\vphi$ is rigid.

Conversely, suppose $\vphi$ is rigid. We will inductively define a $\psi$ and extend $\psi\vphi$ to vertices so as to take larger and larger subgraphs of $G$ isomorphically to their image. As a base case, set $\psi(\hat w) = \hat w$. Setting $\psi\vphi(t(e)) = t(w)$ and $\psi\vphi(o(e)) = o(w)$ extends $\psi\vphi$ to vertices so as to produce an isomorphism between the subgraph of $G$ consisting of $e$ and $t(e), o(e)$, and the subgraph of $H$ consisting of $w$ and $t(w), o(w)$.

Suppose that for a connected subgraph $W$ of $G$, we have specified $\psi$ on edges of $\vphi(W)$ and extended $\psi\vphi$ on vertices so as to be an isomorphism between $W$ and its image. Choose an edge $\ell$ adjacent to a vertex $v$ in $W$. Let its other end be $p$ (which is possibly also in $W$). Without loss of generality, let $\ell$ point at $p$ and suppose that $\vphi$ does not reverse $\ell$. 

By $(1) \Rightarrow (4)$ in \cref{extended torelli} we have that $\vphi_*(P_p) \equiv P_r$ for some vertex $r \in H$. We also have \begin{align*}
	\vphi_*(P_p) &= \vphi_*(P_v) + \vphi_* h_\ell \\
	&= P_{\psi\vphi(v)} + h_{\vphi(\ell)}
\end{align*}
Here $\vphi_*(P_v) = P_{\psi\vphi(v)}$ is a consequence of our inductive hypothesis. Applying $\iota_{\hat w}\inv$ to $P_r - P_{\psi\vphi(v)} = h_{\vphi(\ell)}$, we obtain $r - \psi\vphi(v) \sim t(\vphi(\ell)) - o(\vphi(\ell))$. 

By \cref{series ambiguity}, there is an edge between $\psi\vphi(v)$ and $r$ which is in series with $\vphi(\ell)$. Set $\psi\vphi(\ell)$ to be this edge. We now need to show that setting $\psi\vphi(p) = r$ extends $\psi\vphi$ to an isomorphism on $W \cup \{\ell\} \cup p$. 

The only issue occurs if $p$ was already in $W$, in which case we need to show that $\psi\vphi(p)$ is already $r$. Since by hypothesis $W$ is connected, there must be some simple cycle $Q = u_1, \dots, u_m, \ell$ inside $W \cup \{\ell\}$, which means $\vphi(Q)$ forms a simple unordered cycle in $H$. Since $\psi$ only permutes within series classes, we also have that $\psi\vphi(Q)$ is a simple unordered cycle. There can be only one pair of vertices in $H$ which $\psi \vphi(\ell)$ can go between to complete $\psi\vphi(\{u_1, \dots, u_m\})$ to a simple unordered cycle, specifically the vertices adjacent to exactly one edge in $\psi\vphi(\{u_1, \dots, u_m\})$. It therefore must be that $\psi \vphi(\ell)$ goes between $\psi\vphi(v)$ and $r$, as desired.
\end{proof}

\begin{cor}
If $G$ is $3$-edge connected, then $\vphi$ is rigid if and only if $\vphi$ is an edge isomorphism.
\end{cor}

\begin{cor} \label{orientation is irrelevant}
Rigidity is independent of the choice of base orientations $\Gamma_1$ and $\Gamma_2$.
\end{cor}

\section{Applications, examples, and relations to other results} \label{applications examples and relations to other results}

\subsection{Examples} \label{examples}
We now describe two examples which demonstrate the content of our Torelli results. In both cases we first determine the rigidity divisor of a morphism in $\orCyc$. The first example is rigid, and we indicate the series fixing automorphism our theory says exists. The second example is nonrigid, and we will use our theory to produce an element of the theta divisor which is not preserved.

We begin by extending \cref{computation of pushforward} (see also \cref{computation of pushforward extended}), which concerns the following objects in \orCyc.

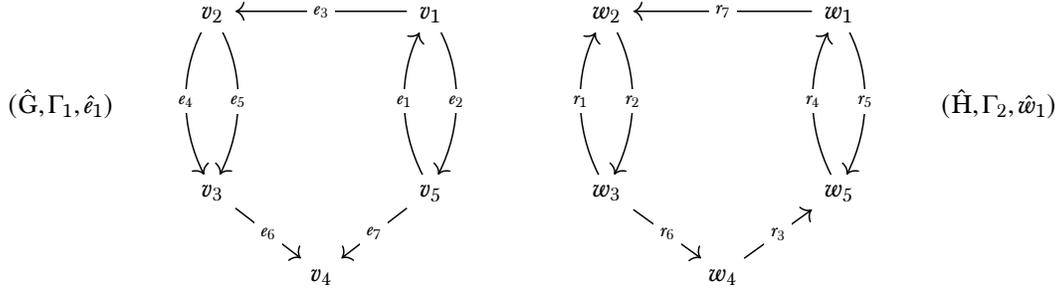
\[\begin{tikzcd}
	& {v_2} && {v_1} && {w_2} && {w_1} \\
	{(\hat G, \Gamma_1, \hat e_1)} &&&&&&&& {(\hat H, \Gamma_2, \hat w_1)} \\
	& {v_3} && {v_5} && {w_3} && {w_5} \\
	&& {v_4} &&&& {w_4}
	\arrow["{e_3}"{description}, from=1-4, to=1-2]
	\arrow["{e_4}"{description}, curve={height=12pt}, from=1-2, to=3-2]
	\arrow["{e_5}"{description}, curve={height=-12pt}, from=1-2, to=3-2]
	\arrow["{e_1}"{description}, curve={height=-12pt}, from=3-4, to=1-4]
	\arrow["{e_2}"{description}, curve={height=-12pt}, from=1-4, to=3-4]
	\arrow["{e_7}"{description}, from=3-4, to=4-3]
	\arrow["{e_6}"{description}, from=3-2, to=4-3]
	\arrow["{r_7}"{description}, from=1-8, to=1-6]
	\arrow["{r_2}"{description}, curve={height=-12pt}, from=1-6, to=3-6]
	\arrow["{r_6}"{description}, from=3-6, to=4-7]
	\arrow["{r_3}"{description}, from=4-7, to=3-8]
	\arrow["{r_4}"{description}, curve={height=-12pt}, from=3-8, to=1-8]
	\arrow["{r_5}"{description}, curve={height=-12pt}, from=1-8, to=3-8]
	\arrow["{r_1}"{description}, curve={height=-12pt}, from=3-6, to=1-6]
\end{tikzcd}\]

There is a cyclic bijection $\vphi$ with $e_i \mapsto r_i$. We previously computed that $\vphi_*: C^1(G, \R) \to C^1(H, \R)$ is given by the matrix \[ \begin{bmatrix}
1 & 0 & 0 & 0 & 0 & 0 & 0 \\
0 & 1 & 0 & 0 & 0 & 0 & 0 \\
0 & 0 & -1 & 0 & 0 & 0 & 0 \\
0 & 0 & 0 & -1 & 0 & 0 & 0 \\
0 & 0 & 0 & 0 & 1 & 0 & 0 \\
0 & 0 & 0 & 0 & 0 & -1 & 0 \\
0 & 0 & 0 & 0 & 0 & 0 & 1 
\end{bmatrix}  \]
We now compute the rigidity divisor. We have \begin{align*}
    \Gamma_1 \xrightarrow{c} v_3 + v_4 \xrightarrow{\iota_{\hat e}} (h_{e_3} + h_{e_5} - h_{e_1} + h_{e_7}, 2) \xrightarrow{(\vphi_* )} (-h_{r_3} + h_{r_5} - h_{e_1} + h_{e_7}, 2)\\
    \Gamma_2 \xrightarrow{c} w_2 + w_5 \xrightarrow{\iota_{\hat w}} (-h_{r_7} + h_{r_5}, 2)
\end{align*}
Hence $\mc E_\vphi = -h_{r_1} + 2h_{r_7} + h_{r_4} + h_{r_6} \equiv -h_{r_1} + h_{r_6} + h_{r_3} + h_{r_4} + h_{r_7} \equiv 0$. \Cref{geometric rigidity} tells us how to construct the series fixing automorphism $\psi$ for which $\psi\vphi$ is an edge isomorphism: $\psi$ swaps $r_3$ and $r_7$ and leaves each other edge in place.

We now examine a nonrigid example. Consider the following objects in \orCyc.

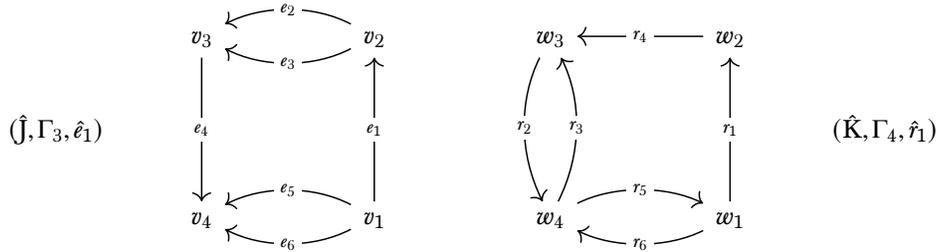
\[\begin{tikzcd}
	& {v_3} && {v_2} && {w_3} && {w_2} \\
	{(\hat J, \Gamma_3, \hat e_1)} &&&&&&&& {(\hat K, \Gamma_4, \hat r_1)} \\
	& {v_4} && {v_1} && {w_4} && {w_1}
	\arrow["{r_6}"{description}, curve={height=-12pt}, from=3-8, to=3-6]
	\arrow["{r_5}"{description}, curve={height=-12pt}, from=3-6, to=3-8]
	\arrow["{r_2}"{description}, curve={height=12pt}, from=1-6, to=3-6]
	\arrow["{r_3}"{description}, curve={height=12pt}, from=3-6, to=1-6]
	\arrow["{r_4}"{description}, from=1-8, to=1-6]
	\arrow["{r_1}"{description}, from=3-8, to=1-8]
	\arrow["{e_1}"{description}, from=3-4, to=1-4]
	\arrow["{e_5}"{description}, curve={height=12pt}, from=3-4, to=3-2]
	\arrow["{e_6}"{description}, curve={height=-12pt}, from=3-4, to=3-2]
	\arrow["{e_2}"{description}, curve={height=12pt}, from=1-4, to=1-2]
	\arrow["{e_3}"{description}, curve={height=-12pt}, from=1-4, to=1-2]
	\arrow["{e_4}"{description}, from=1-2, to=3-2]
\end{tikzcd}\]

We have a morphism $\rho: (\hat J, \Gamma_1, \hat e_2) \to (\hat K, \Gamma_3, \hat r_1)$ which again is defined by $e_i \mapsto w_i$ for each $i$. Our matrix this time is as follows.\[ \begin{bmatrix}
1 & 0 & 0 & 0 & 0 & 0 \\
0 & 1 & 0 & 0 & 0 & 0 \\
0 & 0 & -1 & 0 & 0 & 0 \\
0 & 0 & 0 & 1 & 0 & 0 \\
0 & 0 & 0 & 0 & -1 & 0 \\
0 & 0 & 0 & 0 & 0 & 1 
\end{bmatrix}  \]

We have \begin{align*}
    \Gamma_1 \xrightarrow{c} v_3 + 2v_4 - v_1 \xrightarrow{\iota_{\hat e}} (h_{e_5} + 2h_{e_2} + h_{e_4}, 2) \xrightarrow{(\rho_* )} (-h_{r_5} + 2h_{r_2} + h_{r_4}, 2)\\
    \Gamma_3 \xrightarrow{c} w_3 + w_4 \xrightarrow{\iota_{\hat w}} (2h_{r_4} + h_{r_2}, 2)
\end{align*}

The rigidity divisor is therefore $\mc E_\rho = 3h_{r_2} - h_{r_4} + h_{r_3}$. Divisor calculations show that $\iota_{\hat w}\inv(\mc E_\rho) = w_2 + 3w_3 - 4w_4 \not \sim 0$. 

Since the rigidity divisor is nonvanishing, there must be an element of the discrete theta divisor which $\rho_*$ does not preserve. Our argument tells us how to find a witness to this fact. We have that $\iota_{\hat w}\inv(\mc E_\rho) \sim w_3 + w_1 - 2w_4$. This is a ``$q$-reduced'' form which is convienent for computations \cite{baker_chip-firing_2013}. We observe that $(w_3 + w_1 - 2w_4) + w_4 < w_1 + w_2 + w_3 - w_4$, which is the Chern class of the acyclic orientation specified by $w_4 \prec w_3 \prec w_1 \prec w_2$ (see \cref{ineffective is dominated by acyclic}). So $w_2 + w_4$ is effective, but $(w_2 + w_4) + (w_3 + w_1 - 2w_4)$ is not linearly equivalent to anything effective. We compute that $c\rho_{\mathbb O}\inv \OO(w_2 + w_4) = v_1 + v_4$. Since \begin{align*}\rho_*(S^2_{\hat e_1}(v_1 + v_4), 2) &= \rho_*(\iota_{\hat e_1}(v_1 + v_4)) \\
	&= \rho_*(\iota_{\hat e_1}c\rho_{\mathbb O}\inv \OO(w_2 + w_4)) \\
	&= (w_2 + w_4) + \mc E_\rho \\
	&= (w_2 + w_4) + (w_3 + w_1 - 2w_4) \end{align*} 
	so $S^2_{\hat e_1}(v_1 + v_4)$ is one of the desired elements of the discrete theta divisor.

\subsection{Relation to Whitney's $2$-isomorphism theorem} \label{relation to whitney 2 iso theorem}

In this subsection, we describe a consequence of rigidity which is closely related to the proof of Whitney's 2-isomorphism theorem \cite{truemper_whitneys_1980}, suggesting a conjecture.

\begin{defn}
	A \textit{weak arch} in a graph $G$ is defined in the same way as an arch, except that either the weak arch or its complement must contain three vertices, rather than both being required to have three vertices. We can define Whitney moves on weak arches in precisely the same way as before.
\end{defn}

The relevant difference between an arch and a weak arch is that Whitney moves given by weak arches can be used to arbitrarily reverse the orientations of individual edges. This would trivialize \cref{is orientation preserving possible}, so before we preferred arches, but here our interest is different and our statements are cleaner with weak arches.

\begin{defn}
Let $G$ be a $2$-connective graph and $X_1, \dots, X_n$ be a collection of weak arches with the following properties. \begin{enumerate}
    \item The $X_i$ cover $G$.
    \item Any two distinct $X_i$ intersect at most at their tips.
    \item The graph obtained by quotienting each $X_i$ to an edge is cyclic.
\end{enumerate}
Then the $X_i$ are a \textit{generalized cycle}, and the $X_i$ are a \textit{generalized cycle decomposition} of $G$. We say the decomposition is \textit{in order} or \textit{ordered} if $X_1, \dots, X_n$ enumerates the $X_i$ in the order in which they appear as a cycle. We say that the decomposition is \textit{strong} if each $X_i$ is $2$-connected and the decomposition is in order.
\end{defn}

\begin{remark} \label{removing arches}
A useful property of a generalized cycle decomposition is that removing any weak arch (but not its tips) does not disconnect the graph, but removing any two weak arches (but leaving their tips) does disconnect the graph.
\end{remark}

\begin{prop} \label{refine cycle decomposition}
Let $X_1, \dots, X_n$ be a generalized cycle decomposition in order. Then this decomposition can be refined to a strong decomposition.
\end{prop}

The relevance to cyclic bijections is as follows. The statement here is a little stronger than what appears in \cite{truemper_whitneys_1980}, but follows easily from it.

\begin{theorem} \textit{(Truemper)} \label{cyclic bijections respect strong decompositions}
Let $G$ and $H$ be $2$-connective graphs and let $\vphi: M(G) \to M(H)$ be a cyclic bijection. Let $G$ be equipped with a strong generalized cycle decomposition $X_1, \dots, X_n$. Let $\vphi(X_i)$ denote the smallest subgraph containing the images of all edges in $X_i$. Then there exists a unique reordering $\tau$ of $2, 3, \dots, n$ such that $\vphi(X_1), \vphi(X_{\tau(2)}), \dots, \vphi(X_{\tau(n)})$ is a strong generalized cycle decomposition of $H$ and $\tau(2) \leq \tau(n)$.
\end{theorem}

\begin{proof}
    \cite{truemper_whitneys_1980}.
\end{proof}

In other words, for a each strong generalized cycle decomposition, we have a corresponding decomposition on $H$, and up to a permutation $\vphi$ respects the decomposition structures. Note that $\vphi$ does not necessarily respect the adjacency relations of edges within a particular weak arch; we only know that each weak arch is preserved as a whole. Given the above statement, we are so close to Truemper's proof of Whitney's 2-isomorphism theorem that we are liable to bump our noses on it: with an induction hypothesis on the number of edges, one uses Whitney moves to permute the $X_i$, then for each $X_i$, treats all the other $X_j$ as a single edge and applies the hypothesis.

Our next step is to relate generalized cycles to series classes.

\begin{lemma}
Let $G$ be a $2$-connective and $2$-edge connected graph on at least three vertices, with a strong generalized cycle decomposition $X_1, \dots, X_n$. Pick an edge $r \in G$. At least one of the following holds. \begin{enumerate}
    \item For some $X_i$, we have that $M(r) \subseteq X_i$.
    \item For each $\ell \in \varepsilon(r)$, there is some $X_i$ with $X_i = \{ \ell \}$. Further, no other edge appears as a singleton among the $X_i$. 
\end{enumerate}
\end{lemma}
\begin{proof}
Suppose that not all the edges in $\varepsilon(r)$ are contained in the same $X_i$. Notice that if $\vphi: E(G) \to E(G)$ is any permutation on $\varepsilon(r)$, and the identity on all other edges, then $\vphi$ is a cyclic bijection. By \cref{cyclic bijections respect strong decompositions}, it follows that any $X_i$ which contains an edge in $\varepsilon(r)$ is a singleton, for otherwise we could permute out only a single edge of $\varepsilon(r)$ from $X_i$. Since removing any two singleton arches (but leaving their tips) disconnects $G$, any two edges appearing as singletons in the decomposition must be in series.
\end{proof}

\begin{theorem}
Let $\vphi: (\hat G, \hat e) \to (\hat H, \hat w)$ be a rigid morphism in $\orCyc$ between graphs of genus at least two. Then for every strong generalized cycle decomposition $X_1, \dots, X_n$, the induced reordering $\tau$ of \cref{cyclic bijections respect strong decompositions} is the identity, except that if some series class has its edges as singletons among the $X_i$, then $\tau$ may permute them.
\end{theorem}
\begin{proof}
This follows from \cref{geometric rigidity} and the prior lemma. 
\end{proof}

The converse is intuitively plausible, but seems more difficult to prove. We leave it as the following conjecture.

\begin{conjecture}
Let $\vphi: (\hat G, \hat e) \to (\hat H, \hat w)$ be a morphism in $\orCyc$ between graphs of genus at least two. Suppose that for every strong generalized cycle decomposition $X_1, \dots, X_n$, the induced reordering $\tau$ of \Cref{cyclic bijections respect strong decompositions} is the identity, except that if some series class has its edges as singletons among the $X_i$, then $\tau$ may permute them. Then $\vphi$ is rigid.
\end{conjecture}

\subsection{Main Result As A Strengthening of the Existing Graph Torelli Theorem} \label{relation to cap torelli}

While as stated our result describes properties of a morphism in $\orCyc$, the existing graph Torelli theorem \Cref{capTorelli} begins with an isomorphism of Jacobians and produces isomorphisms of matroids. We now reframe our result in terms of this prior theorem. 

We first sketch how the matroid isomorphisms of the existing graph Torelli theorem are produced (for a description of how to do this in practice see \cite{dancso_construction_2016}). Begin with an isomorphism $\gamma^*: A(H) \to A(G)$. Orient both graphs; in principle the orientations are irrelevant but certain useful orientations simplify the argument. Dualize to obtain an isomorphism $\gamma: J(G) \to J(H)$. This lifts to an isomorphism $H^1(G, \R) \to H^1(H, \R)$, which sends each $h_{\ell}$ to some $\pm h_{a}$. This does not completely specify the desired bijection $\vphi$: by \cref{series functorial ambiguity} it only specifies a map $s$ from edges of $G$ to series classes in $H$. One then proceeds by showing that $s$ is well defined on each series class of $G$, and that for each $e$, the sets $\varepsilon(e)$ and $s(\varepsilon(e))$ actually have the same cardinality. Choosing a bijection for each such pair of series classes then yields a cyclic bijection $\vphi$ lifting $\gamma$ (that is, $\vphi_* = \gamma$).

We can see from this that lifts are completely specified up to series fixing automorphisms. Putting this together with our results yields the following.

\begin{theorem} \label{relation to prior torelli} \textit{(Torelli Theorem for Graphs)}
Let $\mc T$ be the category of isomorphisms between metric tori. Let $\gamma: J(G) \to J(H)$ be a morphism in $\mc T$, where $G$ and $H$ are $2$-connected and $2$-edge connected graphs of genus at least two. Then $\gamma = \vphi_*$ for some $\vphi: (\hat G, \hat e) \to (\hat H, \hat w)$ in $\orCyc$. For each such lift $\vphi$, we have that $\gamma(\theta_{\hat e}) = \theta_{\hat w}$ if and only if $\vphi$ is an edge isomorphism up to a series fixing automorphism of $(\hat H, \hat w)$. If $G$ and $H$ are $3$-edge connected, then there is a unique lift, and preservation of the discrete theta divisor is equivalent to this lift extending to a graph isomorphism.
\end{theorem}

\subsection{Main Result As A Lifting Criterion For Isomorphisms of Graphic Matroids} \label{as matroid lifting}
We can also interpret our main result as telling us that we can use rigidity to check when an isomorphism $\phi: M(G) \to M(H)$ implies the existence of an honest graph isomorphism. 

We will sketch the procedure. Assume $G$ and $H$ are $2$-connected and $2$-edge connected graphs of genus at least two. Consider any orientations $\Gamma_1$ and $\Gamma_2$ on $G$ and $H$, respectively (by \cref{orientation is irrelevant}, the choice is irrelevant). Pick any base edge $\hat e_1$ in $\hat G$ and set $w := \phi(e_1)$ to be the base edge of $\hat H$. Let $\tilde \phi: (G, \Gamma_1, e) \to (H, \Gamma_2, w)$ be the morphism in $\orCyc$ induced by $\phi$ and our choices. 

Now if $\tilde \phi$ is not rigid, the failure of rigidity may be influenced by our choice of $e$, for the requirement that the series fixing automorphism on $M(H)$ fix $w$ can be an obstacle. So let $\varepsilon(w) = \{ w_1, \dots, w_n \}$ and let $\tau_i: (M(H), w) \to (M(H), w_i)$ swap $w$ and $w_i$. Then one of the compositions 
\[ (G, \Gamma_1, e) \xrightarrow{\tilde \vphi} (H, \Gamma_2, w) \xrightarrow{\tilde \tau_i} (H, \Gamma_2, w_i)\]
is rigid if and only if there exists a series fixing automorphism $\psi: M(H) \to M(H)$ such that 
\[ M(G) \xrightarrow{\varphi} M(H) \xrightarrow{\psi} M(H) \]
is an edge isomorphism (note the lack of base elements), and once we detect rigidity, we can reconstruct $\psi$. It would be interesting to see a careful analysis of the computational complexity of this process. 

\begin{conjecture}
When the chosen base edge $e$ is not a series edge, the above process can be carried out in at worst polynomial time in the number of vertices and edges.
\end{conjecture}

\subsection{Fibers of the compactified Torelli map} \label{fibers of the compactified Torelli map}
We will now briefly indicate the context in which our objects arise in algebraic geometry - this is a direction of intended future research for the author. 

Let $A_g$ denote the moduli scheme of principally polarized Abelian varieties of dimension $g$, and $M_g$ denote the moduli scheme of smooth projective curves of genus $g$. The \textit{Torelli map}, denoted $\tau: M_g \to A_g$, is induced by the assignment of curves to their Jacobians. 

It is known that $M_g$ admits a compactification $\overline{M}_g$, the moduli scheme of Deligne-Mumford stable curves. Alexeev constructed and examined a compactification $\overline{A}^{\text{mod}}_g$ of $A_g$, which parameterizes \textit{principally polarized semi-abelic stable pairs} \cite{alexeev_compactified_2004} \cite{alexeev_complete_2004}. Graphs enter naturally into the study of $\overline{M}_g$, because points on the boundary can be described as \textit{stable curves}. Speaking nonrigorously, a stable curve $X$ is a nicely glued collections of irreducible curves $C_1, \dots, C_n$. The $C_i$ can be taken to be vertices of a graph $\Gamma_X$, with edges corresponding to points of intersection. This is referred to as the \textit{dual graph} of the stable curve. Unlike the graphs we have considered above, these may have loops, may have genus one, and may be $1$-connective or $1$-edge connective. To simplify the discussion, we will ignore the existence of loops and graphs of genus one.

We compactify $\tau$, yielding the following commutative diagram.

\[\begin{tikzcd}
	{\overline{M}_g} && {\overline{A}_g^{\text{mod}}} \\
	\\
	{M_g} && {A_g}
	\arrow["\tau", from=3-1, to=3-3]
	\arrow[hook, from=3-1, to=1-1]
	\arrow[hook, from=3-3, to=1-3]
	\arrow["{\overline \tau}", from=1-1, to=1-3]
\end{tikzcd}\]

A natural question to ask is if the fibers of $\overline \tau$ can be characterized combinatorially. Caporaso and Viviani proved in \cite{caporaso_torelli_stable_2010} that the fibers can be described as follows. First, $\tau\inv(A_g)$ consists of $M_g$ and all $X$ such that $\Gamma_X$ is a tree. Second, if $X$ and $X'$ are stable curves such that $\Gamma_X$ and $\Gamma_{X'}$ are $2$-edge connected, then $\overline \tau(X) = \overline \tau(X')$ if and only if $\Gamma_X$ and $\Gamma_{X'}$ are strongly cyclically equivalent (we will define this momentarily). The case in which $\Gamma_X$ is not a tree but has a separating edge is slightly more complex; one removes the separating edge, performs a `stabilization' procedure and arrives at a variant of the above case.

The definition of strong cyclic equivalence is as follows: pick any two edges $e, e'$ in $\Gamma_X$ which are in series. Then $\Gamma_X \bs \{e, e'\}$ has two components $Y_1, Y_2$, each of which is an arch in $\Gamma_X$. They are thus both candidates for Whitney moves. Let us call a Whitney move obtained in such a way \textit{admissible}. Then we deem $\Gamma_X$ and $\Gamma_{X'}$ to be \textit{strongly cyclically equivalent} if there is a cyclic bijection between them, such that in the classical result \cref{whitney 2 iso}, the Whitney moves can all be taken to be admissible. It is easy to see that admissible Whitney moves are sufficient to permute edges within their series class as we like. It is also easy to produce an example of a nonrigid admissible Whitney move. 

If $X$, $X'$ are points in the same fiber of the Torelli map, then choosing an isomorphism of their associated principally polarized semi-abelic stable pairs is almost sufficient to specify a (strong) cyclic bijection between them. The only ambiguity arises from what are, in our terminology, series fixing automorphisms. Since rigidity is independent of such automorphisms, it is independent of this last choice. The author suspects that with care, examining how rigidity varies given choices of isomorphisms of principally polarized semi-abelic stable pairs will shed new light on the structure of the fibers of the Torelli map. 

\section{SageMath software tools} \label{software}
The author has developed software tools within SageMath \cite{sagemath}, located at \url{https://github.com/SeaGriff/Graph-Orientations-and-Divisors}, which implement much of the material described both here and in \cite{backman_riemann-roch_2017}. The code is designed to facilitate passing data between the existing SageMath implementations of the discrete Picard group and matroids, and implementations of $\mathbb O(G)$ and morphisms in $\orCyc$. The useful functions include:
\begin{enumerate}
    \item Taking Chern classes of orientations.
    \item Lifting divisors to partial orientations, or certifying this cannot be done, as in \cite{backman_riemann-roch_2017}.
    \item The unfurling algorithm of \cite{backman_riemann-roch_2017}.
    \item The modified unfurling algorithm of \cite{backman_riemann-roch_2017}.
    \item Oriented Dhar's algorithm, as in \cite{backman_riemann-roch_2017}.
    \item The computation of signs $\sgn_\vphi(\ell)$ of morphisms in $\orCyc$.
    \item Computation of rigidity divisors of morphisms in $\orCyc$.
    \item Construction of edge isomorphisms from rigid morphisms in $\orCyc$.
\end{enumerate}
The author has been unable to locate prior implementations of the generalized cycle-cocycle system or the above functions.

\section{Acknowledgements}
This material is based upon work supported by the National Science Foundation under Grant No. 1701659.

The author wishes to thank David Jensen for many helpful comments on the presentation and organization of the paper.

\nocite{Arkor_quiver_2021}

\printbibliography

\end{document}